\RequirePackage{fix-cm}
\documentclass[smallextended]{svjour3}       
\smartqed  
\usepackage{graphicx}
\usepackage[utf8x]{inputenc}
\usepackage{amsfonts,vmargin,amssymb}
\usepackage{amsmath,amssymb}
\usepackage{tikz}
\newcommand{\syst}[2]{\e{\left\{\tab{#1}{#2} \right. }}
\newcommand{\systnn}[2]{\enn{\left\{\tab{#1}{#2} \right. }}
\newcommand{\Ok}{{\Omega_k}}

\newcommand{\bOk}{{\partial \Omega_k}}
\newcommand{\skj}{{\Sigma_{kj}}}
\newcommand{\sjk}{{\Sigma_{jk}}}
\newcommand{\Gk}{{\Gamma_k}}
\newcommand{\trr}{(-\partial_\nu + i \gamma)}
\newcommand{\trs}{( \partial_\nu + i \gamma)}
\newcommand{\unNh}{[\![ 1, N_h ]\!]}

\newcommand{\e}[1]{\begin{equation}#1\end{equation}}

\newcommand{\enn}[1]{\begin{equation*}#1\end{equation*}}
\newcommand{\tab}[2]{\begin{array}{#1}#2\end{array}}

\newcommand{\dx}{\partial_x}
\newcommand{\ddx}{\partial_x^2}
\newcommand{\dy}{\partial_y}
\newcommand{\ddy}{\partial_y^2}
\newcommand{\lu}[2]{\lambda_{#1,#2}}
\newcommand{\luoz}{\lambda_{1,0}}
\newcommand{\luzo}{\lambda_{0,1}}
\newcommand{\luij}{\lambda_{i,j}}
\newcommand{\dux}{\partial_x}
\newcommand{\duy}{\partial_y}

\begin{document}

\title{Interpolation properties of generalized plane waves
}


\author{Lise-Marie Imbert-G\'erard
}


\institute{Lise-Marie Imbert-G\'erard \at
              first address \\
              Tel.: +123-45-678910\\
              Fax: +123-45-678910\\
              \email{imbertgerard@cims.nyu.edu}           
}

\date{Received: date / Accepted: date}

\maketitle

\begin{abstract}
This paper aims at developing new shape functions adapted to smooth vanishing coefficients for scalar wave equation. It proposes the numerical analysis of their interpolation properties. The interpolation is local but high order convergence is shown with respect to the size of the domain considered. The new basis functions are then implemented in a numerical method to solve a scalar wave equation problem with a mixed boundary condition. The order of convergence of the method varies linearly with the one of the interpolation.
\end{abstract}

\section{Introduction}
\label{intro}
This paper focuses designing Generalized Plane Waves (GPW) to approximate smooth solutions $u\in \mathcal C^\infty(\Omega)$ of the model problem 
\e{\label{eq:Helm}
-\Delta u + \beta u = 0,
\quad \text{ in } \Omega \subset \mathbb R^2,
}
where $\beta$ is in $\mathcal C^\infty(\Omega)$.  
This time-harmonic equation, generally called the scalar wave equation, models for instance the acoustic pressure describing the behavior of sound in matter or a polarized electromagnetic wave propagating in an isotropic medium.
If $\beta=-\omega^2$, $\omega\in \mathbb R$, the equation is classically named the Helmholtz equation and is still the subject of recent research, see for instance \cite{Euan}. If $\beta<0$ is non constant, this is a simple model of wave propagation in an inhomogeneous medium. If $\beta>0$, it models an absorbing medium and the partial differential equation is coercive.
The applications considered here include both propagative and absorbing medium, as well as smooth transitions in between them, i.e. respectively $\beta>0$, $\beta<0$ and $\beta=0$.

Several types of numerical methods are used for the simulation of wave propagation. Classical finite element methods applied to such problems are known to be polluted by dispersion, see \cite{poll}.
An alternative is to consider approximation methods based on shape functions that are local solutions of the homogeneous equation: this justifies the development of Trefftz-based methods, first introduced in \cite{trefftz}, that rely on solutions of the homogeneous governing domain equation: information about the problem is embedded in the finite dimension basis functions set. 
The present work originated from the idea to apply such a method to a problem modeled by \eqref{eq:Helm} in which the coefficient is likely to vanish: shape functions adapted to this problem are here designed and studied. See the previous work \cite{LMIGBD} for the physical motivation of the problem.
Refer to \cite{plu2} and references therein for more recent developments of these Trefftz-based methods, and to \cite{git,HMP13} for applications linked to one specific method, the so-called Ultra Weak Variational Formulation (UWVF). The method coupling the latter to the adapted shape functions is the topic of \cite{LMIGBD}, and the present work includes numerical results for the $h$ convergence of the coupled method.

The novelty in the present paper lies in the smooth feature of the coefficient $\beta$ of the governing domain equation \eqref{eq:Helm} and on the explicit procedure proposed to design the corresponding shape functions. This work can be compared to recent works that focus on non polynomial methods for smooth varying coefficients, see for instance \cite{Bet,Tez}.
The design of shape functions adapted to smooth and possibly vanishing coefficients that is the core of this work starts from mimicking the equation
\enn{
(-\Delta + \beta)e^{I\omega \overrightarrow k \cdot \overrightarrow x} 
=\left( -(I\omega\|\overrightarrow k \|)^2-\omega^2\right)e^{I\omega \overrightarrow k \cdot \overrightarrow x} = 0,}
that shows that classical plane waves functions $e^{I\omega \overrightarrow k \cdot \overrightarrow x}$ are exact solutions of \eqref{eq:Helm} when $\beta=-\omega^2$ is constant and negative.

The case of a piecewise constant coefficient is addressed for example in \cite{cd98,nc}, and the more general case of a smooth coefficient is generally approximated by a piecewise constant coefficient on each cell of the mesh. 
A very simple extension of classical plane waves for a positive or negative constant coefficient would be to consider at a point $G=(x_G,y_G)$ the shape function 
\e{
\label{eq:simplGPW}
\varphi(x,y)=\exp \left(\sqrt{\textrm{sgn}(\beta(G))}\sqrt{|\beta(G)|}\left((x-x_G)\cos \theta + (y-y_G) \sin \theta\right) \right),
} 
where the parameter $\theta$ represents the direction of the plane wave. 
Indeed, the case $\beta(G)<0$ corresponds to the classical plane wave whereas the case $\beta(G)>0$ corresponds to a complex wavenumber. This choice will provide a tool to extend the interpolation results cited previously. 
Remark that if the coefficient $\beta(G)<0$ goes to zero then the corresponding classical plane waves functions, generates by equi-spaced directions $\theta$, tend not to be independent anymore.  

Consider here the case of a general smooth coefficient $\beta$. 
Typically, in the case $\beta(x,y)=x$ the Airy functions $Ai$ and $Bi$ are solutions ; 
however, in the general case there is no exact analytic solution known. Indeed, as was explained in Section 2.1 of \cite{LMIGBD}, no exponential of a polynomial can solve a generic scalar wave equation. So the idea is to generalize the classical plane wave function as an approximated solution of the initial equation, in the following sense : design $\varphi = e^{P}$, with a complex polynomial $\displaystyle P(x,y)=\sum_{i=0}^{{\rm d}P} \sum_{j=0}^{{\rm d}P-i} \lu{i}{j} (x-x_G)^i (y-y_G)^j$ such that
\enn{(-\Delta  + \beta ) e^{P} =\big(-(\ddx P + (\dx P)^2 +\ddy P + (\dy P)^2) + \beta\big) e^{P}}
is \emph{locally} small. 
More precisely, the generalized plane wave described in this paper will be designed to satisfy \emph{locally}
\e{\label{-lapl+al}
 -\left(\ddx P + (\dx P)^2 +\ddy P + (\dy P)^2\right)(x,y) + \beta(x,y) = O \left(\| (x,y)-(x_G,y_G) \|^q\right),
}
up to a given order $q\in\mathbb N$, satisfying $q\geq 1$.
In this process, the degree and coefficients of $P$ will be chosen to satisfy \eqref{-lapl+al}, and this approximation identity is equivalent to canceling the $q(q+1)/2$ coefficients of lower degree terms in the Taylor expansion of its left hand side. It provides a system which unknowns are  the coefficients of $P$ and which size does depend on $q$. This resulting system can be either underdetermined or overdetermined, depending on the degree of $P$, denoted ${\rm d}P$, with respect to the value of $q$. A specific procedure will be described in order to obtain a square invertible system. Its main feature is based on the idea of generalizing the classical plane wave function as displayed for $\beta(G)<0$ in \eqref{eq:simplGPW}, by setting $P(x,y)=\sqrt{\beta(G)}((x-x_G)\cos \theta + (y-y_G)\sin\theta) +$ {\it higher order terms}. As a result the reasoning leading to interpolation properties of such new functions is built on the simpler case of classical plane waves. 

In addition to the approximation order $q$, the general design procedure proposed in this paper involves two parameters:
\begin{itemize}
\item a parameter $\theta$ corresponding to the direction of a classical plane wave,
\item a parameter $N\neq 0$ , where $N/I$ can be interpreted as the local wave number of a classical plane wave.
\end{itemize}
These parameters are used to set $(\luoz,\luzo)=N (\cos \theta,\sin\theta)$. It justifies the name given to the new shape functions: generalized plane waves. Additional constraints greatly simplify both the numerical computations 
and the analysis of the method. See Definition \ref{def:GPW}.

For a given value of $(N,\theta)$, formula \eqref{eq:TExy} together with the Definitions \ref{def:AN} and \ref{def:CN} provide an explicit function $\varphi=e^P$, which satisfies the approximation identity \eqref{-lapl+al}. Varying $\theta$ then provides different functions $\varphi$, only as long as $N\neq 0$. This condition $N\neq 0$ is mandatory to define a set of linearly independent shape functions: they then form a basis of an approximation space $\mathcal E (G,N,p,q)$.
\begin{definition}\label{def:E(G,p)}
Suppose $N\in\mathbb C$ such that $N\neq 0$ and $p\in \mathbb N$ is such that $p\geq3$. Consider then for all $l \in [\![1,p]\!]$
\begin{itemize}
\item $\theta_l=2\pi (l-1)/p$ a direction, all directions being equi-spaced, 
\item $(\luoz^l,\luzo^l)=N (\cos \theta_l,\sin\theta_l)$ the corresponding coefficients of the degree one terms,
\item $\varphi_l $ the corresponding generalized plane wave.
\end{itemize}
 The set of $p$ shape functions denoted $\mathcal E (G,N,p,q)$ is defined by $\left\{ \varphi_l \right\}_{l\in[\![1,p]\!]}$.
\end{definition}
 This set of $p$ basis functions is meant to approximate the solution of scalar wave equation. The goal of the theoretical part of this paper is to prove high order approximation properties on such sets of basis functions, provided that enough basis functions are used with respect to the approximation parameter $q$. Note that the design process does not involve the number $p$ of basis functions. See the hypothesis of the following claim to quantify the relation between the parameters $p$ and $q$.
  
The parameter $N$ is then the main degree of freedom to be fixed to define explicitly the approximation space $\mathcal E(G,N,p,q)$. As will be detailed in Section \ref{sec:des}, two different choices will be considered in this paper. 
A first choice is comprised of setting $N=\sqrt{\beta(G)}$, see Definition \ref{def:AN}. It gives a direct generalization of a classical plane waves, since in this case $\beta(G)<0$, so that $\sqrt{-\beta(G)}$ is the local wave number. However, this choice is local since $N$ does depend on $G\in \Omega$, and cannot be used if $\beta(G)=0$: it is a classical problem in low frequency regime since the linear independence of such shape functions is damaged, see \cite{LFHipt,imbert}.

To overcome this limitation and consider the stationary limit case, a second possibility is to choose one constant and non zero value for $N$: it will not depend on $G$ anymore. The choice $N=I=\sqrt{-1}$ is proposed to address the case $\beta(G)=0$ and ensure the desired interpolation property.

\begin{claim}
Denote by $n\in\mathbb N$ an interpolation parameter and by $G\in \Omega$ a point in $\mathbb R^2$. Consider a smooth solution $u\in\mathcal C^{n+1}(\Omega)$ of scalar wave equation \eqref{eq:Helm}. Set then
\begin{itemize}
\item $N\neq 0$ the degree of freedom in the design process,
\item $q\geq n+1$ the order of approximation in \eqref{-lapl+al},
\item $p=2n+1$ the number of basis functions in $\mathcal E (G,N,p,q)$.
\end{itemize}   
There exists  an approximation $u_a\in Span\ \mathcal E (G,N,p,q)$ of order $n+1$ of $u$ in the following sense: there is a constant $C_{N,\Omega}$ such that for all $M=(x,y)\in\Omega$
 \enn{
\left\{ 
\begin{array}{l}
\left| u(M)-u_a(M)\right| \leq C_{N,\Omega} |M-G|^{n+1} \left\| u \right\|_{\mathcal C^{n+1}} ,\\
\left\| \nabla u(M)-\nabla u_a(M)\right\| \leq C_{N,\Omega} |M-G|^{n} \left\| u \right\|_{\mathcal C^{n+1}}. 
\end{array}
\right.
}
\end{claim}
The behavior of the constant $C_{N,\Omega}$ as $N$ goes to zero is commented in Subsection \ref{ssec:estCN}. 
It suggests the need for a parameter $N$ that is bounded away from zero, see Subsections \ref{ssec:ff} and \ref{ssec:sf}.

There are two main streams in proving such interpolation results for Helmholtz equation that have been developed in the literature. One of them is based on Vekua theory, which was first translated into English in \cite{hen} for functions in $\mathbb R^2$. A more recent introduction to the topic can be found in \cite{BetT}. Theoretical studies based on this technical tool can be found in \cite{melenk2}, and more recently in \cite{Moiola}. In the latter, the case of Helmholtz equation with constant coefficient is explicitly studied and interpolation properties are obtained with explicit dependence with respect to the parameters. However, even if this theory is powerful, in the case of a smooth coefficient it gives no explicit estimates with respect to the different parameters. On the other hand, another method using Taylor expansions was proposed in \cite{cd98}. Since the design of solutions developed in this paper is based on Taylor expansions as well, this second method will be used here.

Section \ref{sec:des} describes precisely the design process, and develops some properties of the resulting approximated functions. It defines two different ways of defining the new functions, called normalizations. Section \ref{sec:interp} focuses on the proof of Theorem \ref{th:u-ua}, considering these two normalizations as well. A last section presents a numerical application with a method based on the generalized plane wave basis functions and some numerical results. The numerical test cases are chosen to consider problems linked with reflectometry, a radar diagnostic technique for fusion plasma, see \cite{LMIGBD} for more details.


{\bf Notation.} The symbol $\partial_z$ represents the partial derivative with respect to the variable $z$. The symbol $I$ represents the complex number $I=\sqrt {-1}$ to avoid any confusion with the summation index $i$.

\section{Design and properties of a shape function}\label{sec:des}
This section concerns the design of shape functions that are locally approximated solutions of the scalar wave equation \eqref{eq:Helm}. 
The point $G=(x_G,y_G)\in\Omega$ is fixed, and the design process, based on Taylor expansions, depends on that point $G$, the degree of freedom $N$, the number $p$ of shape functions and the order of approximation $q$. 
This order of approximation will satisfy $q\geq 1$. The case $q=1$ corresponds to the simplest generalization of plane waves described previously in \eqref{eq:simplGPW}. The design of the polynomial $P$ starts with the choice of its degree, and then focuses on computing its coefficients to satisfy \eqref{-lapl+al}.

Some properties of two different types of shape functions follow. They are meant to be used in the proof of Theorem \ref{th:u-ua}.

\subsection{Design procedure}\label{ssec:DP}

\begin{definition}\label{def:PL} Denote by $P$ a bivariate polynomial.
The polynomial $(\Delta e^P)/e^P$ will be denoted $P_\Delta$, so that
$$P_\Delta = \left(\ddx P + (\dx P)^2 +\ddy P + (\dy P)^2\right) $$
\end{definition}
In order to satisfy the local approximation \eqref{-lapl+al}, the design is based on a non linear system on the coefficients of $P$ that arises from considering the Taylor expansions in scalar wave equation of $\beta-P_\Delta$ up to the order $q$. Thanks to Definition \ref{def:PL} it reads 
\e{\label{eq:a-PL}
\beta(x,y)-P_\Delta(x,y)=O(\|(x,y)-(x_G,y_G)\|^q).
}

The procedure includes choosing the degree of the polynomial and giving an explicit expression to compute the coefficients of the polynomial. These two choices are not independent. A precise analysis of equation \eqref{eq:a-PL} leads to choosing the degree of $P$ such that the computation of the coefficients appears to be straightforward. 
\begin{remark}
Since the constant coefficient $\lambda_{0,0}$ does not appear in \eqref{eq:a-PL}, it is set to zero. This   will simplify all the upcoming computations. Moreover, the fact that it does not depend either on $G$ or on $\beta$ prevents any blow up of the corresponding shape function since then $\varphi(G) = e^{\lambda_{0,0}}$ is constant.
\end{remark}

The system to be solved to ensure that equation \eqref{eq:a-PL} holds has:
\begin{itemize}
\item $N_{un} = \frac{({\rm d}P+1)({\rm d}P+2)}{2}-1$ unknowns, namely the coefficients of $P$ except $\lambda_{0,0}$,
\item $N_{eq} = \frac{q(q+1)}{2}$ equations, corresponding to the cancellation of the terms of degree lower than $q$ in the Taylor expansion of $\beta-P_\Delta$.
\end{itemize}
As a result the system is overdetermined if ${\rm d}P< q$, and in such a case the existence of a solution is not guaranteed. The idea is then to find the smallest value of ${\rm d}P\geq q$ that would provide an invertible system. 

The case ${\rm d}P = q$ is more intricate than the next one, since the $q$ equation stemming from the terms of degree $q-1$ have no linear term. It does not - in general - lead to a convenient invertible system. Indeed, in a such case, the system is underdetermined however there is no straightforward way to obtain an invertible system, because of the nonlinearity.

As for the case ${\rm d}P = q+1$, the system is underdetermined and the number of additional equations to be imposed to get a square system is $N_{un}-N_{eq}=2q+2$. Moreover, since
\enn{
\beta (x,y) = \sum_{ (i,j) \slash 0\leq i+j \leq q-1 }
 \frac{ \dx^i\dy^j\beta (x_G, y_G)}{i!j!} (x-x_G)^i (y-y_G)^j +  O\left( \| (x,y)-(x_G,y_G) \|^q \right),
}
then the $N_{eq}$ equations of the system that come from \eqref{eq:a-PL} actually reads 
\e{\label{eq:TExy}
\tab{l}{
\forall (i,j) \text{ s.t. } 0\leq i+j \leq q-1, \\
\displaystyle  \frac{\dux^i \duy^j \beta(G)}{i!j!} 
= (i+2)(i+1)\lu{i+2}{j}+ (j+2)(j+1) \lu{i}{j+2} 
\\ \displaystyle \phantom{\frac{\dux^i\beta(G) \duy^j \beta(G)}{i!j!} = } 
+\sum_{k=0}^{i} \sum_{l=0}^{j} (i-k+1)(k+1) \lu{i-k+1}{j-l} \lu{k+1}{l}
\\ \displaystyle \phantom{\frac{\dux^i\beta(G) \duy^j \beta(G)}{i!j!}  = } 
+\sum_{k=0}^{j} \sum_{l=0}^{i} (j-k+1)(k+1) \lu{i-l}{j-k+1} \lu{l}{k+1}.
}}
As a consequence, to obtain an invertible system the choice proposed in this paper is to fix the set of coefficients $\left\{ \lu{i}{j}, i\in \{0,1\}, j \in [\![ 0, q+1-i ]\!] \right\}$. Thus this choice corresponds to the $2q+3$ additional constraints that, together with equations \eqref{eq:TExy}, form a square system. 

\begin{proposition}\label{lem:U}
The system described by \eqref{eq:TExy} together with the additional constraints of fixing the elements of $\left\{ \lu{i}{j}, i\in \{0,1\}, j \in [\![ 0, q+1-i ]\!] \right\}$ has a unique solution, given by
\e{\label{eq:IF}
\tab{l}{
\forall (i,j) \text{ s.t. } 0\leq i+j \leq q-1, \\
\displaystyle \lu{i+2}{j} 
= \frac1{(i+2)(i+1)}\Bigg(\frac{\dux^i \duy^j \beta(G)}{i!j!} -  (j+2)(j+1) \lu{i}{j+2} 
\\ \displaystyle \phantom{\frac{\dux^i\beta(G) \duy^j \beta(G)}{i!j!} = } 
-\sum_{k=0}^{i} \sum_{l=0}^{j} (i-k+1)(k+1) \lu{i-k+1}{j-l} \lu{k+1}{l}
\\ \displaystyle \phantom{\frac{\dux^i\beta(G) \duy^j \beta(G)}{i!j!}  = } 
-\sum_{k=0}^{j} \sum_{l=0}^{i} (j-k+1)(k+1) \lu{i-l}{j-k+1} \lu{l}{k+1}\Bigg).
}}
\end{proposition}

\begin{proof}
For any given set of coefficients $\left\{ \lu{i}{j}, i\in \{0,1\}, j \in [\![ 0, q+1-i ]\!] \right\}$, the existence and uniqueness of a solution of \eqref{eq:TExy} stems directly from the induction relation \eqref{eq:IF}.
 See Figure \ref{fig:indl}.
\begin{figure}
\begin{center}
 \includegraphics[width=0.45\textwidth]{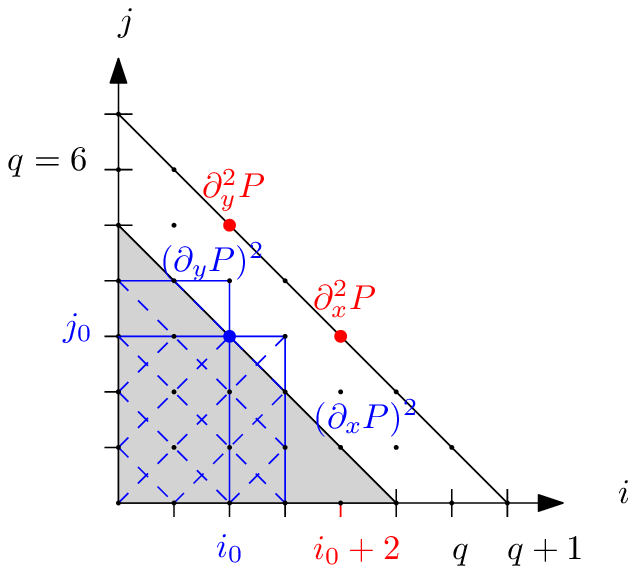}
 \includegraphics[width=0.45\textwidth]{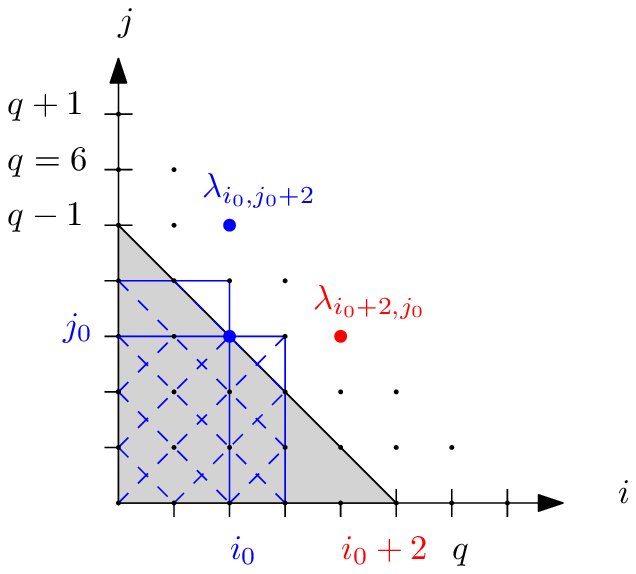}
 \caption{
For a given $(i_0,j_0)$, the left part of the figure shows the contributions from $P_\Delta$ to the $x^{i_0}y^{j_0}$ term in $\beta-P_\Delta$. The right part shows that $\lu{i_0+2}{j_0}$ can be explicitly expressed as long as $\lu{k}{l}$ are known for all $k\leq i_0+1$ and $l\leq {\rm d}P-2-k$.}
\label{fig:indl}
\end{center}
\end{figure}
\end{proof}

In this paper the set of coefficients $\left\{ \lu{i}{j}, i\in \{0,1\}, j \in [\![ 0, q+1-i ]\!] \right\}$ will be fixed in the following way.
\begin{definition}\label{def:GPW}
Denote by $q\in\mathbb N^*$ the approximation order, by $\theta \in \mathbb R$ and $N\in\mathbb C$ such that $N\neq 0$. A generalized plane wave is a function $\varphi=e^P$, with $\displaystyle P=\sum_{(i,j)\backslash 0\leq i+j\leq q+1} \luij (x-x_G)^i (y-y_G)^j$ which coefficients satisfy
\begin{itemize}
\item $(\luoz,\luzo)=N (\cos \theta,\sin\theta)$ as described in the introduction,
\item $\lambda_{0,0}=0$ to avoid any blow up of the shape function linked to the exponential,
\item $ \lu{i}{j}=0$ for $i\in\{0,1\}$ and $1<i+j\leq q+1$.
\end{itemize}
 and the induction formula \eqref{eq:IF}.
\end{definition}
The last item is the simplest possible choice and is meant to simplify both the numerical computations - by a substantial decrease of basic operations necessary to evaluate a shape function - and the analysis of the method.

\begin{remark}[Other possible choices]
Other choices to obtain an invertible system would give the same theoretical results. For instance choosing to fix  $\{ \lu{i}{j},j\in\{0,1\},i\in [\![ 0, q+1-j ]\!]\}$ is possible as well. But numerically, as will be seen later on, there is no evidence of the lack of symmetry with respect to the two space variables.
\end{remark}

\subsection{A fundamental property of a generalized pane wave}
Since the design and the interpolation study are based on different Taylor expansions, the derivatives of the shape function $\varphi$ are important quantities.
Both
\begin{enumerate}
\item[$\bullet$] the coefficients $\luij$s defining a shape function $\varphi$ 
\item[$\bullet$] the derivatives of $\varphi$ 
\end{enumerate} 
are here expressed as polynomials with two variables with respect to $(\luoz,\luzo)$.
The following Lemma \ref{lem:deglu} and Proposition \ref{prop:derphi} give a description of these quantities with respect to the only non zero coefficients fixed as constraints, namely $(\luoz,\luzo)$.

\begin{lemma}\label{lem:deglu}
For a given set of coefficients $\left\{ \lu{i}{j}, i\in \{0,1\}, j \in [\![ 0, q+1-i ]\!] \right\}$, the set of  coefficients $\{\lu{i}{j}, 0\leq i+j\leq q+1\}$ that are the unique solution of \eqref{eq:TExy} from Lemma \ref{lem:U} can be described as polynomials with two variables in $(\luoz,\luzo)$ as follows.
\begin{equation}\label{tab:lus}
\left\{
\begin{array}{l}
\forall i\geq 2
\\ \lu{i}{j}  \textrm{ is of total degree at most } i-2 . 
\end{array}
\right.
\end{equation}
\end{lemma}

The following proof relies on a close examination of the induction formula \eqref{eq:TExy}, considered as polynomial with two variables, namely $(\luoz,\luzo)$. The idea is to track the terms with higher degree.

\begin{proof}
Because of the null coefficients, formula \eqref{eq:TExy} for $i=0$ and $i=1$ reads
\syst{rll}{\label{sys:lus}
 \beta(G) &= 2\lu{2}{0} + (\luoz)^2+(\luzo)^2,&\\
\displaystyle \frac{\dy^j \beta(G)}{j!} &\displaystyle  = 2\lu{2}{j} & \forall j>0,\\
\displaystyle \dx \beta(G)  &= 6 \lu{3}{0}+4\lu{2}{0} \luoz, &\\
\displaystyle \frac{\dx \dy^j \beta(G)}{j!} &= 6 \lu{3}{j}+4\lu{2}{j} \luoz & \forall j>0.
}
Then \eqref{tab:lus} for $i=2$ stems from point 
\ref{norm:1} of the normalization. Indeed for $j=0$ the sum $(\luoz)^2+(\luzo)^2$ does not depend on $(\luoz,\luzo)$ themselves but only on $N$. 
Afterwards \eqref{tab:lus} for $i=3$ is clear from \eqref{sys:lus}.

Now set $i\geq2$ and suppose that the statement \eqref{tab:lus} holds true for all $ \tilde i \in [\![ 3,i+1 ]\!]$. Then, isolating $\lu{i+2}{j}$ in \eqref{eq:TExy}, the highest possible degree of each term is 
\begin{itemize}
\item[$\bullet$] ${i-2}$ for the term in $\lu{i}{j+2}$,
\item[$\bullet$] $ (i-1)+1$ for the term in $\lu{i+1}{j}\luoz$,
\item[$\bullet$] $(i-k-1)+(k-1)$ for the terms in $\lu{i-k+1}{j-l}\lu{k+1}{l}$ with $k\neq0$ and $k\neq i$,
\item[$\bullet$] $(i-2)+1$ for the term in $\lu{i}{j+1}\luzo$,
\item[$\bullet$] $(i-l-2)+(l-2)$ for the term in $\lu{i-l}{j-k+1}\lu{l}{k+1}$ with $l\neq0$ and $l\neq i$, 

note that $\lu{i-l}{j-k+1}\lu{l}{k+1}=0$ with $l\neq 1$ and $l\neq i-1$ because of the point 
\ref{norm:2} of the normalization.
\end{itemize}
As a consequence the terms with higher degree appearing in the expression of $\lu{i+2}{j}$ have degree at most equal to $i$. It completes the proof of \eqref{tab:lus} for $i>2$ by induction.
\end{proof}

\begin{proposition}\label{prop:derphi}
Suppose $\theta \in \mathbb R$ and $N\in\mathbb C$ is such that $N\neq 0$. Consider a shape function $\varphi = e^{P}$ constructed in Subsection \ref{ssec:DP}. Then for all $(i,j) \in \mathbb N^2$ such that $i+j\leq q+1$ there is a complex polynomial $R_{i,j}$ such that its total degree satisfies ${\rm d} R_{i,j} \leq i-2$ and such that 
\e{\label{eq:derphi}
 \dx^i\dy^j \varphi (G)
 = (\luzo)^j (\luoz)^i+R_{i,j}(\luoz,\luzo).
}
The coefficients of $R_{i,j}$ only depend on $N$ and on the derivatives of $\beta$.
\end{proposition}

\begin{remark}\label{rem:NU}
Since $(\luoz)^2+(\luzo)^2$ is fixed, none of the polynomial expressions that are at stake can be unique. 
For instance, any occurrence of $(\luoz)^2$ could be replaced by $N^2-(\luzo)^2$ which would change the term of higher degree. 
This is the reason why $R_{i,j}$ is not unique: see Subsection \ref{ssec:alg} for a different point of view. However, formula \eqref{eq:IF} from Proposition \ref{lem:U} gives an explicit procedure for the computation of all $\luij$s: this is the crucial point that will be used for practical implementation.
\end{remark}

One could have expected the degree of $R_{i,j}$ to be smaller than $i+j-1$. The fact that it does actually not depend on $j$ is due to the choice of $\{ \lu{i}{j},i\in\{0,1\},i+j> 1\}$ to be zero. The fact that it is smaller than $i-2$ is due to the fact that the degree of $\lu{2}{j}$ is $0$, since $(\luoz)^2+(\luzo)^2=N^2$ is constant with respect to $\luzo$ and $\luoz$. See Definition \ref{def:GPW}.

\begin{proof}
Applying the chain rule introduced Appendix \ref{app:bivFDB} to $\varphi=e^{P}$ one gets for all $ (i,j) \in \mathbb N^2$,
\enn{
\dx^i\dy^j \varphi (G)
 = i!j!\sum_{\mu=1}^{i+j} \sum_{s=1}^{i+j} \sum_{p_s((i,j),\mu)} \prod_{l=1}^s \frac{(\lu{i_l}{j_l})^{k_l}}{k_l!},
}
where $p_s((i,j),\mu)$ is the set of partitions of $(i,j)$ with length $\mu$:
\enn{
 \left\{ (k_l,(i_l,j_l))_{l\in [\![ 1,s ]\!]}:k_l\in\mathbb N^*, 0\prec (i_1,j_1)\prec \dots \prec(i_l,j_l), \sum_{l=1}^s k_l = \mu,  \sum_{l=1}^s k_l(i_l,j_l) = (i,j) \right\}.
}
Now consider such a partition to be given and focus on the degree of the corresponding product term, namely $\displaystyle \prod_{l=1}^s (\lu{i_l}{j_l})^{k_l}$. Thanks to Lemma \ref{lem:deglu} one can split this product into different terms regarding their degree as polynomials with respect to $(\luoz\luzo)$. As a result, since $\displaystyle Deg \ \prod_{l=1}^s (\lu{i_l}{j_l})^{k_l}=\sum_{l=1}^s k_l Deg\ \lu{i_l}{j_l}$, this quantity is also at most equal to
\e{\label{eq:sumdeg}
\sum_{i_l=0,j_l=1} k_l j_l+\sum_{i_l=1,j_l=0} k_l i_l+\sum_{i_l=2} k_l\cdot 0+\sum_{i_l\geq3} k_l (i_l-2) ,
}
where the two first sums contain at most one term each. 

Obviously the leading term in $\dx^i\dy^j \varphi (G)$ is $(\luzo)^j(\luoz)^i$, it corresponds to the partition $(i,j)=j(0,1)+i(1,0)$. Indeed, 
as long as a partition contains at least one term such that $i_l\geq 2$, the resulting degree computed from \eqref{eq:sumdeg} will contain at least one term $k_l\cdot0 $ or $k_l(i_l-2)$, and any of them is at most $k_l(i_l+j_l)-2$; as a consequence the degree computed in \eqref{eq:sumdeg} is then strictly lower than $\displaystyle \sum_{l=1}^s k_l(i_l+j_l)-2= i+j-2$.

Since the product term corresponding to the partition $j(0,1)+i(1,0)$ is $(\luzo)^j(\luoz)^i/(j!i!)$ it completes the proof.
\end{proof}

\subsection{A more algebraic viewpoint}\label{ssec:alg}
This paragraph presents a more algebraic point of view on 
Remark \ref{rem:NU}.

Suppose $N\in\mathbb C$ is such that $N\neq 0$. The value of $(\luoz,\luzo)$ gives that $P_{N}=(\luoz)^2+(\luzo)^2-{N}^2$ satisfies $P_{N}=0$ for the $p$ different functions of $\mathcal E(G,N,p,q)$. From then on, considering other quantities as polynomials with two variables in $(\luzo,\luoz)$ is in fact computing in the quotient ring $\mathbb C [\luoz,\luzo] /(P_{N})$ of $\mathbb C [\luoz,\luzo]$ modulo the ideal generated by $P_{N}$. For instance, the system \eqref{sys:lus} reads 
\systnn{rll}{
\lu{2}{0} &\displaystyle =\frac{\beta(G)-{N}^2}{2}&  (P_{N}), \\
\displaystyle \lu{2}{j} &\displaystyle  =  \frac{\dy^j \beta(G)}{2(j!)}&  (P_{N}),\ \forall j>0,\\
\displaystyle \lu{3}{0}&\displaystyle  =  \frac{\dx \beta(G) -2\luoz(\beta(G)-{N}^2) }{6} &(P_{N}),\\
\displaystyle \lu{3}{j}&\displaystyle  =  \frac{\dx \dy^j \beta(G) }{6(j!)}+2  \frac{\dy^j \beta(G)}{j!} \luoz &  (P_{N}),\ \forall j>0.
}
Of course in this quotient ring, each equivalence class has an infinite number of elements, and all the computations of the previous subsection are performed on elements of these classes. 
Thus any equality applies to all the elements of the same class. Note that since the ring considered here is the ring of polynomials with two variables, there is no such thing as the Euclidean division. 
As a result there is nothing like a canonical element of a class used for computations. One can easily see that for $q\geq 4$
\enn{\tab{rl}{
\dx^4\dy \varphi (G)& =  (\luoz)^4 (\luoz) +2\dy \beta(G)\left( (\luoz)^2 - (\luzo)^2\right)+2\dx\beta(G) \luzo\luoz+ 2\dx\dy\beta(G)\luoz\\ & \quad+(-3\dy^2\beta(G)+\dx \beta(G))\luzo  -\dy^3\beta(G) + \dx^2\dy\beta(G), \\
 &=  (\luoz)^4 (\luoz) +2\dy \beta(G)\left( (\luoz)^2 + (\luzo)^2\right)+2\dx\beta(G) \luzo\luoz+ 2\dx\dy\beta(G)\luoz\\ & \quad+(-3\dy^2\beta(G)+\dx \beta(G))\luzo  -\dy^3\beta(G) + \dx^2\dy\beta(G)-2\beta(G)\dy \beta(G),
}}
which gives two possible $R_{4,1}\in \mathbb C[\luoz,\luzo]$ satisfying \eqref{eq:derphi} in Proposition 
\ref{prop:derphi}.


\subsection{First family of generalized plane waves}
\label{ssec:ff}

The first type of shape functions corresponds to a {\it local} choice since it does depend on $G\in\mathbb R^2$.
\begin{definition}\label{def:AN}
The $\beta$-normalization is defined by choosing $N=\sqrt{\beta(G)}$ in Definition \ref{def:GPW}, which means setting
\begin{enumerate}
\item\label{norm:1}
$(\luoz,\luzo)=\sqrt{\beta(G)} (\cos \theta,\sin\theta)$. 
\item\label{norm:2}
$\{ \lu{i}{j},0\leq i+j\leq q+1,i+j\neq 1\}$ are set to zero.
\end{enumerate}
\end{definition}

\begin{remark}[Back to classical Plane Waves from the $\beta$-normalization]\label{rem:BtCPW}
The fact that the quantity $(\luoz)^2+(\luzo)^2$ is equal to $\beta(G)$ however gives that the value of $\beta(G)$ does actually never appear in the expression of the other coefficients explicitley, but only in product terms involving $\luoz$ or $\luzo$. One can easily check by induction that all the terms appearing in formula \eqref{eq:IF} are then linear combinations of the derivatives of $\beta$. 
As a consequence, for $\beta=-\omega^2<0$ and for any $q\geq 1$, all the coefficients $\luij$ such that $i>1$ are actually zero, which means that the corresponding function $\varphi=e^P$ is nothing more than a classical plane wave.

As remarked in the introduction, it is also obvious that for $q=1$ this new shape function is again nothing more than a classical plane wave as long as $\beta<0$. This case $q=1$ corresponds to the classical fact of approximating a smooth coefficient by its piecewise constant value at the center of the cells.
\end{remark}

\subsection{Second possibility}
\label{ssec:sf}
The fact is that since the terms $(\luoz,\luzo)$ of the $\beta$-normalization are proportional to the square root of $\beta$, they will tend to zero with $\beta$. A theoretical estimate displayed in Subseciton \ref{ssec:estCN} will justify the need for a second normalization. Moreover numerical results show that it causes severe damaging to the conditioning of the discrete UWVF problem if $\beta$ tends to zero. As a consequence, a second normalization is considered, with a {\it global} choice of $N$ independent from $G\in\mathbb R^2$. 
\begin{definition}\label{def:CN}
The constant-normalization is defined by choosing $N=i$ in Definition \ref{def:GPW}, which means setting 
\begin{enumerate}
\item
$(\luoz,\luzo)=i (\cos \theta,\sin\theta)$. 
\item
$\{ \lu{i}{j},0\leq i+j\leq q+1,i+j\neq 1\}$ are set to zero.
\end{enumerate}
It corresponds to $N=i$. 
\end{definition}

\begin{remark}[Classical Plane Waves and the constant-normalization]\label{rem:BtCPWCST}
In order to balance Remark \ref{rem:BtCPW}, note that since the constant-normalization does not depend on $\beta$ it arises that for $\beta=-\omega^2\neq -1$ the term $\beta(G)$ appears in higher order terms. For instance it is clear that
\enn{\lu{2}{0}=\frac{1+ \beta(G)}{2}.}
As a consequence, neither when $\beta(\neq -1)<0$ is constant nor when $q=1$ the shape function designed from the constant-normalization can be a classical plane wave.
\end{remark}

\section{Interpolation}
\label{sec:interp}

The interpolation properties of the set $\mathcal E (G,N,p,q)$ are defined in Theorem \ref{th:u-ua}. This section is devoted to the proof of this result, which states that, in order to approximate to a given order $n+1$ the solution of the scalar wave equation \eqref{eq:Helm} around a point $G$, a sufficient number $p=2n+1$ of basis functions together with a sufficient approximation parameter $q=n+1$ are required. The gradient of the solution is then approximated to the order $n$.

\subsection{Theoretical result}
This subsection focuses on the interpolation property of the set of basis functions $\mathcal E (G,N,p,q)$. 
The sketch of the proof is inspired by the one developed by Cessenat in \cite{cd98}, but it is adapted to the generalized plane wave basis functions. 
 Note that the application to the UWVF is postponed to Section \ref{apps}.

\begin{definition}\label{def:Mns}
Suppose that $N\in\mathbb C$ is such that $N\neq 0$, $n\in\mathbb N^*$ and $G\in\mathbb R^2$. For all $l\in \mathbb N$ such that $1\leq l\leq n$ consider the direction $\theta_l=2\pi(l-1)/(2n+1)$, the generalized plane wave $\varphi_l$, $\kappa=-iN\in\mathbb C^*$ and the function
$$e_l(x,y)=e^{I\kappa\left((x-x_G)\cos \theta_l + (y-y_G) \sin \theta_l\right)},$$ 
which is a classical plane wave if $N \in i\mathbb R$.
The $(n+1)(n+2)/2\times(2n+1)$ matrices $M_n^C$ and $M_n$ are defined as follows: for all $(k_1,k_2)\in\mathbb N^2$, such that $k_1+k_2\leq n$
\begin{equation*}
\left\{
\begin{array}{l}
\left(M_n^C \right)_{\frac{(k_1+k_2)(k_1+k_2+1)}{2}+k_2+1,l} = \frac{\dx^{k_1}\dy^{k_2}e_l (G)}{k_1 ! k_2!}, \\
 \left(M_n \right)_{\frac{(k_1+k_2)(k_1+k_2+1)}{2}+k_2+1,l} =\frac{\dx^{k_1}\dy^{k_2}\varphi_l (G)}{k_1 ! k_2!} .
\end{array}
\right.
\end{equation*}
Their $l$th columns contain respectively the Taylor expansion coefficients of the functions $e_l$ and $\varphi_l$.
\end{definition}
\noindent For instance, one has $
M_1=
\begin{pmatrix}
\varphi_1 (G) &\varphi_2(G) & \varphi_3 (G) \\
\dx\varphi_1 (G) &\dx\varphi_2(G) & \dx\varphi_3 (G) \\
\dy\varphi_1 (G) &\dy\varphi_2(G) &\dy \varphi_3 (G) 
\end{pmatrix}$, $M_1^C=\begin{pmatrix}
1&1&1\\
I\kappa\cos\theta_1&I\kappa\cos\theta_2 &I\kappa\cos\theta_3\\
I\kappa\sin\theta_1 & I\kappa\sin\theta_2 &I\kappa \sin\theta_3
\end{pmatrix}
$ and
\begin{equation*}
M_2=
\begin{pmatrix}
\varphi_1 (G) &\varphi_2(G) & \varphi_3 (G)&\varphi_4(G) & \varphi_5 (G)  \\
\dx\varphi_1 (G) &\dx\varphi_2(G) & \dx\varphi_3 (G) &\dx\varphi_4(G) & \dx\varphi_5 (G) \\
\dy\varphi_1 (G) &\dy\varphi_2(G) &\dy \varphi_3 (G)  &\dy\varphi_4(G) &\dy \varphi_5 (G) \\
\dx^2\varphi_1 (G)/2 &\dx^2\varphi_2(G)/2 & \dx^2\varphi_3 (G)/2 &\dx^2\varphi_4(G)/2 & \dx^2\varphi_5 (G)/2 \\
\dx\dy\varphi_1 (G) &\dx\dy\varphi_2(G) &\dx\dy \varphi_3 (G) &\dx\dy\varphi_4(G) &\dx\dy \varphi_5 (G) \\
\dy^2\varphi_1 (G)/2 &\dy^2\varphi_2(G)/2 &\dy^2 \varphi_3 (G)/2  &\dy^2\varphi_4(G)/2 &\dy^2 \varphi_5 (G)/2
\end{pmatrix}.
\end{equation*}

The rank of the matrix $M_n^C$ is computed in Lemma \ref{lem:rkMnC}, which profits from the fact that the result proved by Cessenat and Despr\'es in \cite{cd98} for $\kappa>0$ is actually still valid for $\kappa \in \mathbb C^*$. 
The proof of Theorem \ref{th:u-ua} relies on Lemma \ref{lem:MnMnC} that explicits the link between the matrix $M_n^C$ and the corresponding matrix $M_n$ built with the generalized plane waves. 

\begin{lemma}\label{lem:rkMnC}
Suppose that $N\in\mathbb C$ is such that $N\neq 0$, $n\in\mathbb N^*$ and $G\in\mathbb R^2$. 
There are two matrices: a rectangle matrix $P_n$ only depending on $\beta(G)$ and a square invertible matrix $S_n$ only depending on the directions $\theta_l$ such that $S_n=P_n \cdot M_n^C$. Moreover $rk(M_n^C)=2n+1$.
\end{lemma}

\begin{proof}
Consider $M_n^C$ be the matrix introduced in Definition \ref{def:Mns} so that for all $ (k_1,k_2)\in\mathbb N^2$, such that $k_1+k_2\leq n$
\enn{
\left(M_n^C \right)_{\frac{(k_1+k_2)(k_1+k_2+1)}{2}+k_2+1,l} =  \frac{\dx^{k_1} \dy^{k_2}e_l(G)}{k_1 ! k_2!}
= \frac{(I\kappa)^{k_1+k_2}}{k_1 ! k_2!} \cos^{k_1} \theta_l \sin^{k_2} \theta_l.
}
Define for all $ k\in[\![0,n]\!]$
 $$
(S_n)_{n\pm k+1,l}=\frac{1}{(I\kappa)^k}\left(\partial_x\pm I \partial_y  \right)^k e_l(G)=\frac{k!}{(I\kappa)^k}\sum_{s=0}^{k}\frac{(\pm I)^s\dx^{(k-s)} \dy^{s}e_l(G)}{(k-s)! s!}.
$$
 Thanks to the definition of $M_n^C$ one can check that 
$$
 (S_n)_{n\pm k+1,l}=\frac{k!}{(I\kappa)^k}\sum_{s=0}^{k}(\pm I )^s (M_n^C)_{\frac{((k-s)+s)((k-s)+s+1)}{2}+s+1,l},
$$
 so that $S_n$ is a $(2n+1)\times(2n+1)$ matrix that is a linear transform of $M_n^C$. More precisely, define $P_n$ as a $(2n+1)\times\frac{(n+1)(n+2)}{2}$ matrix such that 
$$(P_n)_{l,\frac{k(k+1)}{2}+s+1}=k!(\pm I)^s/(I\kappa)^k.$$ Then $S_n=P_n\cdot M_n^C$. As a consequence, $rk(M_n^C)\geq rk(S_n)$.

The rank of $S_n$ is now to be evaluated thanks to the definition of the plane waves $e_l$. Since $e_l(x,y)=e^{(i\kappa)\left((x-x_G)\cos \theta_l + (y-y_G) \sin \theta_l\right)}$ then
\enn{
\left(\partial_x\pm I \partial_y  \right)^k e_l = (i\kappa)^k (\cos \theta_l \pm I \sin \theta_l)^k e_l.
}
Consider that $z_l=\cos \theta_l + I \sin \theta_l=(\cos \theta_l - I \sin \theta_l)^{-1}$ because $|z_l|=1$, and since $e_l(G)=1$ it yields
\enn{
\left(\partial_x\pm I \partial_y  \right)^k e_l (G)= (I\kappa)^k(z_l)^{\pm k} \Rightarrow (S_n)_{n\pm k+1,l}=(z_l)^{\pm k}.
}
Thus $S_n$'s columns are proportional to the one of a VanDerMonde matrix and 
\enn{ det \ S_n = \prod_{i=1}^n z_i^{-n}\prod_{i<j}(z_i-z_j).}
From the choice of $\theta_l$s, for all $i\neq j$: $z_i\neq z_j$ so that $S_n$ is invertible and $rk(M_n^C)\geq rk(S_n)=2n+1$. Since $$rk(M_n^C) \leq \min \left(2n+1, \frac{(n+1)(n+2)}{2}\right)=2n+1$$ the proof is then completed.
\end{proof}

\begin{lemma}\label{lem:MnMnC}
Suppose that $N\in\mathbb C$ is such that $N\neq 0$, $n\in\mathbb N^*$ and $G\in\mathbb R^2$. 
Consider $\mathcal E(G,N,p,q)$ introduced in Definition \ref{def:E(G,p)}, together with $M_n$ and $M_n^C$ introduced in Definition \ref{def:Mns}. 
Then there is a lower triangular matrix $L_n$, which diagonal coefficients are all equal to $1$ and which other coefficients are linear combinations of the derivatives of $\beta$ evaluated at $G$, such that
\e{\label{eq:MnMnC}
M_n=L_n\cdot M_n^C.}
As a consequence $rk(M_n)=rk(M_n^C)$ and both $\|L_n\|$ and $\|(L_n)^{-1}\|$ are bounded by a constant only depending on $\beta$.
\end{lemma}

The following proof is straightforward considering the feature of the derivatives of $\varphi_l$ described in Proposition \ref{prop:derphi}.

\begin{proof}
From \eqref{eq:derphi} there exists a polynomial $R_{i,j}\in\mathbb C[X,Y]$ with $Deg\ R_{i,j}\leq i-2$ such that
\e{\label{eq:MnMnCp}
\forall (i,j) \in \mathbb N^2,\dx^{i}\dy^{j}\varphi_l (G) = \dx^{i}\dy^{j}e_l (G) + R_{i,j} (\dx e_l (G),\dy e_l (G) ).
}
The coefficients of $R_{i,j}$ do not depend on the basis function considered, but only depends on $\beta$ and its derivatives evaluated at $G$. By construction of the classical plane wave $e_l$, one has 
\enn{
\left\{\tab{rl}{
\dx^{k}\dy^{m}e_l (G) &=\left(\dx e_l (G)\right)^k\left(\dy e_l (G)\right)^m,
\\ &= (I\kappa)^{k+m} \cos( \theta)^k (i\sin(\theta))^m.
}\right.}
 The numbering of the rows in matrices $M_n^C$ and $M_n$ is set up such that the derivatives of smaller order appear higher in the matrix, which proves \eqref{eq:MnMnC}. 
Indeed \eqref{eq:MnMnCp} shows that any coefficient of $M_n$ is the sum of the corresponding coefficient in $M_n^C$ plus a linear combination - which coefficients do not depend on the column that is considered but only on $\beta$ and its derivatives evaluated at $G$ - of terms that appear higher in the corresponding column of $M_n$.

The rank of $M_n$ is then equal to the rank of $M_n^C$, and $\|L_n\|$ and $\|(L_n)^{-1}\|$ do only depend on the coefficients of $R_{i,j}$. As a result they do not depend on the basis functions but only on the coefficient $\beta$ and its derivatives at $G$.
\end{proof}

\begin{theorem}\label{th:u-ua}
Suppose that $n\in \mathbb N$ and that $u$ is a solution of scalar wave equation \eqref{pbi} belongs to $\mathcal C^{n+1}$. Consider then $q\geq n+1$, $p=2n+1$, and $\mathcal E (G,N,p,q)$ introduced in Definition \ref{def:E(G,p)}. 
Then there are a function $u_a\in Span\ \mathcal E(G,N,p,q)$ depending on $\beta$ and $n$, and a constant $C_{N,\Omega}$ depending on $N$, $\beta$ and $n$ such that for all $M\in \mathbb R^2$
\e{\label{eq:gradumua}
\left\{ 
\begin{array}{l}
\left| u(M)-u_a(M)\right| \leq C_{N,\Omega} |M-G|^{n+1} \left\| u \right\|_{\mathcal C^{n+1}(\Omega)} ,\\
\left\| \nabla u(M)-\nabla u_a(M)\right\| \leq C_{N,\Omega} |M-G|^{n} \left\| u \right\|_{\mathcal C^{n+1}(\Omega)}. 
\end{array}
\right.
}
\end{theorem} 
\begin{proof}
The idea of the proof is to look for $\displaystyle u_a= \sum_{l=1}^{2n+1} x_l \varphi_l$ by fitting its Taylor expansion to the one of $u$. This will be done by solving a linear system concerning the unknowns $(x_l)_{l\in[\![1,2n+1]\!]}$.

Since $u$ belongs to $\mathcal C^{n+1}$ and for all $l\in [\![1,2n+1]\!]$ the basis function $\varphi_l$ belongs to $\mathcal C^\infty$, their Taylor expansions read for all $M=(x,y) \in \Omega$
\enn{
\left| u(x,y) - \sum_{m=0}^n \sum_{k_1+k_2=m} B_{k_1k_2}x^{k_1} y^{k_2} \right| \leq C |M-G|^{n+1} \|u\|_{\mathcal C^{n+1}},
}
\enn{
\left| \varphi_l(x,y) - \sum_{m=0}^n \sum_{k_1+k_2=m} M^l_{k_1k_2}x^{k_1} y^{k_2} \right| \leq C |M-G|^{n+1} \|\varphi_l\|_{\mathcal C^{n+1}},
}
where for the sake of simplicity $M^l_{k_1k_2}$ stands for the coefficient of $M_n$ that corresponds to the term $\dx^{k_1} \dy^{k_2} \varphi_l /(k_1!k_2!)$, namely the coefficient $(M_n)_{\frac{(k_2+k_1)(k_2+k_1+1)}{2}+k_2+1,l}$, and in the same way $B_{k_1,k_2}$ stands for the coefficient $(B_n)_{\frac{(k_2+k_1)(k_2+k_1+1)}{2}+k_2+1}$. The system to be solved is then
\begin{equation*}
\left\{
\begin{array}{l}
 \text{Find } (x_l)_{l\in[\![1,2n+1]\!]} \in \mathbb C^{2n+1} \text{ s. t.}\\
\displaystyle \sum_{l=1}^{2n+1} M^l_{k_1,k_2} x_l = B_{k_1,k_2},\ \forall m \in[\![0,n]\!], \ \forall (k_1,k_2)\in [\![0,n]\!]^2 \text{ s. t. } k_1+k_2=m.
\end{array}
\right.
\end{equation*}

In order to study the system's matrix, the equations depending on $(k_1,k_2)$ have to be numbered: they will be considered with increasing $m=k_1+k_2$, and with decreasing $k_1$ for a fixed value of $m$. 
Defining the corresponding vector $B_n\in \mathbb C^{\frac{(n+1)(n+2)}{2}}$, together with the unknown $X^n=(x_1,x_2,\cdots,x_{2n+1})\in \mathbb C^{2n+1}$, the system now reads
\begin{equation*}
\left\{
\begin{array}{l}
 \text{Find } X^n \in \mathbb C^{2n+1} \text{ such that}\\
   M_n \cdot X^n = B_n
\end{array}
\right.
\end{equation*}
where $M_n\in \mathbb C^{\frac{(n+1)(n+2)}{2}\times (2n+1)}$ is the matrix from Definition \ref{def:Mns}.

Since the system is not square, there is a solution if and only if $B_n\in Im(M_n)$. 

{\bf i)} The technical point is to prove that $rk(M_n)=2n+1$. It is straightforward from Lemmas \ref{lem:MnMnC} and \ref{lem:rkMnC}.

{\bf ii)} There exists a subset $K \subset  \mathbb C^{\frac{(n+1)(n+2)}{2}}$ such that $Im(M_n)\subset K$ and $B_n\in K$. This subspace $K$ is built from the fact that the basis functions are designed to fit the Taylor expansion of the scalar wave equation:
\enn{
K:= \left\{ (C_{k_1,k_2}) \in  \mathbb C^{\frac{(n+1)(n+2)}{2}}, \forall(k_1,k_2)\in\mathbb N^2, k_1+k_2\leq n-2, \phantom{\sum_{j=0}^{k_1} \frac{\dx^i \dy^j \beta(G)}{i!j!} C_{k_1-j,k_2} (k_1+1)}\right.
}
\e{\label{eq:defK}
\left. (k_1+1)(k_1+2)C_{k_1+2,k_2} +(k_2+1)(k_2+2)C_{k_1,k_2+2}=\sum_{i=0}^{k_1}\sum_{j=0}^{k_2} \frac{\dx^{i} \dy^{j} \beta(G)}{i!j!} C_{k_1-i,k_2-j} \right\}
}
All basis functions $\varphi_l$, $l\in [\![1,2n+1]\!], $ satisfy $(-\Delta+\beta)\varphi_l=(-P_{\Delta,l}+\beta)\varphi_l$.
From the  equation \eqref{eq:a-PL} with $q\geq n+1$, it is then straightforward to see that $Im(M_n)\subset K$. The fact that $B_n\in K$ simply stems from plugging the Taylor expansions of $u$ and $\beta$ into scalar wave equation.

{\bf iii)} The dimension of $K$ defined by \eqref{eq:defK} is $2n+1$. Indeed, one can check - using the same numbering as previously for the equations - that $K$ is defined by $n(n+1)/2$ linearly independent relations on $  \mathbb C^{\frac{(n+1)(n+2)}{2}}$, so that its dimension is $(n+1)(n+2)/2-n(n+1)/2$.

As a consequence, from the solution to the system  $M_n \cdot X^n = B_n$ that now is known to exist, one can define $\displaystyle u_a= \sum_{l=1}^{2n+1} x_l \varphi_l$. Thanks to that definition and to the Taylor expansions of $u$ and the $\varphi_l$s it yields
\enn{
\left| u(M)-u_a(M)\right| \leq C |M-G|^{n+1} \left( \left\| u \right\|_{\mathcal C^{n+1}}+\left\| u_a \right\|_{\mathcal C^{n+1}}\right).
}
Moreover one has the identity $X^n=(S_n^C)^{-1} P_n^C (L_n)^{-1}B_n$, where $(S_n^C)^{-1} P_n^C$ is bounded from above by $\displaystyle \sup_{l\in [\![1,2n+1]\!]} \| e_l \|_{\mathcal C^{n+1}} $, see Lemma \ref{lem:rkMnC}, $(L_n)^{-1}$ is bounded from above by a constant depending only on $\beta$ and its derivatives from Lemma \ref{lem:MnMnC}, and $B_n$ is bounded by $\|u\|_{\mathcal C^{n+1}}$. 
Since for all $ l\in [\![1,2n+1]\!] $ it yields $|x_l|\leq C_{N,\Omega}\|u\|_{\mathcal C^{n+1}}$, it turns out to be the first part of \eqref{eq:gradumua}:
\enn{
\left| u(M)-u_a(M)\right| \leq C_{N,\Omega}(2n+2) |M-G|^{n+1} \left\| u \right\|_{\mathcal C^{n+1}} .
}

At last, the second part of \eqref{eq:gradumua} stems from taking the Taylor Lagrange formula of the gradient of $u-u_a$, up to the order $n$, since 
$$\displaystyle \sum_{m=0}^n \sum_{k_1+k_2=m}\left(  B_{k_1k_2}(x-x_G)^{k_1} (y-y_G)^{k_2}  - \sum_{l=1}^{2n+1}\left(x_lM^l_{k_1k_2}(x-x_G)^{k_1} (y-y_G)^{k_2}\right)
 \right) =0.$$ That is: for all $M=(x,y)\in \Omega$ there are $\zeta_1$, $\zeta_2$ in $\mathbb R^2$ on the segment line between $M$ and $G$ such that
\begin{equation*}
\left\{
\begin{array}{l}
 \displaystyle \dx(u - u_a)(x,y) =
 \sum_{l=0}^{n} \frac{\dx^{l+1}\dy^{n-l}(u-u_a)(\zeta_1)}{l!(n-l)!}  (x-x_G)^{l}  (y-y_G)^{n-l},
\\\displaystyle \dy(u - u_a) (x,y)=
 \sum_{l=0}^{n} \frac{\dx^{l}\dy^{n-l+1}(u-u_a)(\zeta_2)}{l!(n-l)!}  (x-x_G)^{l}  (y-y_G)^{n-l}
\end{array}
\right.
\end{equation*}
which indeed leads to the desired inequality.
\end{proof}

\begin{remark}
Some comments on the hypothesis on $q$ are to be found in the next paragraph.
\end{remark}

\subsection{ Estimate of $C_{N,\Omega}$ with respect to $N\rightarrow 0$}
\label{ssec:estCN}
 Because the Taylor expansion actually reads
$$(u - u_a)(x,y) = \sum_{j=0}^q \left( \frac{\dx^j\dy^{q-j} u(G)}{j!(q-j)!} - \sum_{l=1}^p x_l \frac{\dx^j\dy^{q-j} \varphi_l(G)}{j!(q-j)!} \right)(x-x_G)^{j}  (y-y_G)^{q-j}+ O\left( |M-G|^{q+1} \right),$$
one can see that if $\displaystyle C_{N,\Omega} = \sum_{l=1}^p x_l  \sum_{j=0}^q\frac{\dx^j\dy^{q-j} \varphi_l(G)}{j!(q-j)!} $ blows up when $N$ goes to zero, then so does $C_{N,\Omega}$. 

As displayed in Proposition \ref{prop:derphi}, each $\dx^j\dy^{q-j} \varphi_l(G)$ term is a polynomial with respect to $(\luoz,\luzo)$ which higher degree term is $\luoz^j\luzo^{q-j}$, that is to say that each derivative term from $C_{N,\Omega}$ is a polynomial with respect to $N$ which higher degree term is $\cos (\theta_l)^j (i\sin (\theta_l))^{q-j} N^q$. As a result these terms are not zero and tend to zero as $N$ tends to zero, at most as $N^q$.

On the other hand since $X^n$ satisfies $M_n^C \cdot X^n = (L_n)^{-1} B_n$, one can describe the asymptotic behavior of the $x_l$s with respect to $N$ as follows. Consider the $(2n+1)\times (2n+1)$ square system defined as: for all $k\in\mathbb N$ such that $0\leq k \leq n$
$$ \sum_{l=1}^n \left( \sum_{s=0}^k C_k^s (\pm I)^s\dx^{(k-s)} \dy^s e_l(G) \right) x_l = \big( (L_n)^{-1} B_n \big)_{n\pm k+1}.$$
The right hand side of this system is independent on $N$ while its determinant is $N^{n(n+1)} det(S_n)$. As a result the coefficients $x_l$ behave (as long as it is non zero) as 
$C_{N,\Omega}/N^{n(n+1)}$ 
as $N$ tends to zero.

As a consequence, $C_{N,\Omega}$ as well as $C_{N,\Omega}$ blow up at least as $1/N^{(n-1)(n+1)}$ as $N$ goes to zero.

\subsection{Numerical validation}
In order to validate the Theorem \ref{th:u-ua}, each of the numerical validation case is computed, for a given value of $n$, setting $q=n+1$ and $p=2n+1$. The test case considered is $\beta(x,y)=x-1$, to approximate the exact solution $u_e(x,y)=Airy(x)e^{iy}$. See \cite{LMIGBD} for the physical motivation of this test case: its main interest is that the coefficient vanishes along the line $x=1$, which represents a plasma cut-off that reflects incoming waves.

Of course since the theoretical results give local approximation properties, the validation procedure itself will be local as well. As stated in the theorem $u_e$ can be approximated by a function $u_a$ that belongs to the approximation space $\mathcal E (G,N,p,q)$, space that is built with either the $\beta$-normalization or the constant-normalization. 

The idea is to follow the error $\max \left|u-u_a\right|$ on disks with decreasing radius $h$ in order to observe the order of convergence with respect to $h$. Several different cases are proposed to validate the theoretical order of convergence, and an additional case concerns the behavior of the basis functions designed with the $\beta$-normalization as the approximation point gets closer to the cut-off.

\subsubsection{In the propagative zone}
The point $G=[-3,1]$ is in the propagative zone. Then concentric disks are centered on $G$ with radius $h=1/2^k$, increasing the value of $k$. Following the theorem, the expected order of convergence is $n+1$. 

Figure \ref{fig:cvP} displays computed convergence results that fit perfectly the theoretical result. A set of $p=11$ classical Plane Waves is used as a control case, since $p=11$ is the highest number of basis functions used in the different cases with the Generalized Plane Waves. Note that machine precision is reached in some cases.
\begin{figure}
\begin{center}
\includegraphics{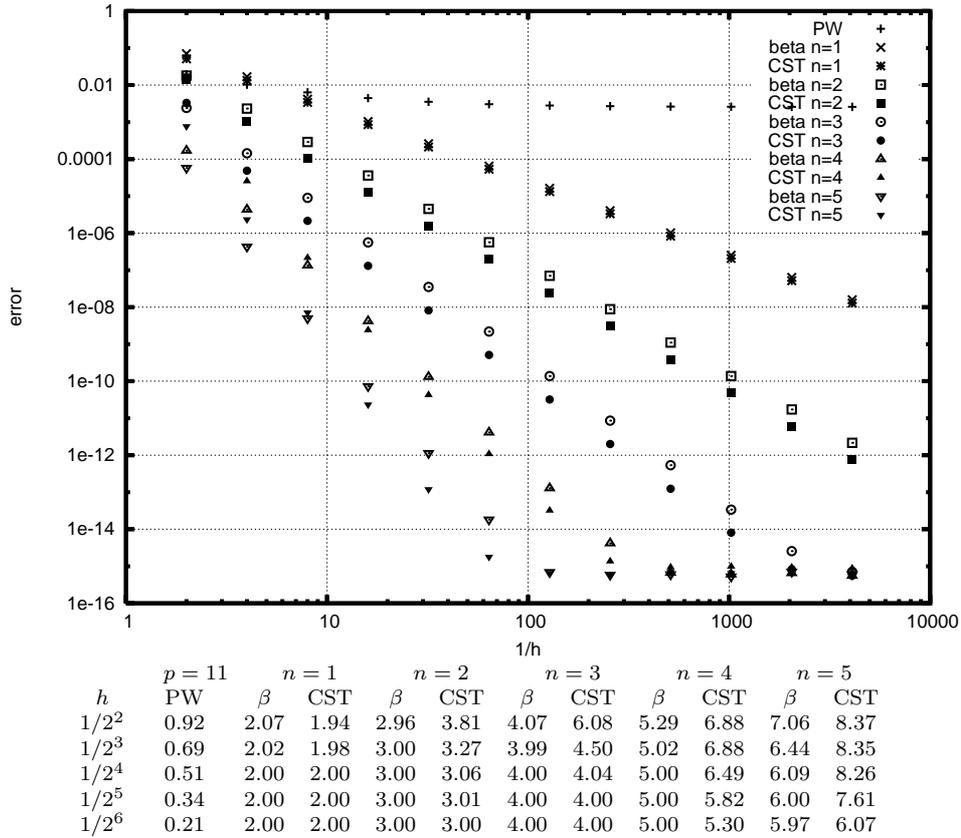}
\begin{tabular}
{@{}c@{\hspace{.3cm}}c@{\hspace{.3cm}}c@{\hspace{.3cm}}c@{\hspace{.3cm}}c@{\hspace{.3cm}}c@{\hspace{.3cm}}c@{\hspace{.3cm}}c@{\hspace{.3cm}}c@{\hspace{.3cm}}c@{\hspace{.3cm}}c@{\hspace{.3cm}}c@{}}
 \multicolumn{1}{c}{} & \multicolumn{1}{c}{$p=11$} & \multicolumn{2}{c}{$n=1$}& \multicolumn{2}{c}{$n=2$}& \multicolumn{2}{c}{$n=3$}& \multicolumn{2}{c}{$n=4$}& \multicolumn{2}{c}{$n=5$} \\
$h$ & PW & $\beta$ & CST & $\beta$ & CST & $\beta$ &CST & $\beta$ & CST & $\beta$ & CST\\ 
1/$2^{2}$ & 0.92 & 2.07 & 1.94 & 2.96 & 3.81 & 4.07 & 6.08 & 5.29 & 6.88 & 7.06 & 8.37 \\ 
1/$2^{3}$ & 0.69 & 2.02 & 1.98 & 3.00 & 3.27 & 3.99 & 4.50 & 5.02 & 6.88 & 6.44 & 8.35 \\ 
1/$2^{4}$ & 0.51 & 2.00 & 2.00 & 3.00 & 3.06 & 4.00 & 4.04 & 5.00 & 6.49 & 6.09 & 8.26 \\ 
1/$2^{5}$ & 0.34 & 2.00 & 2.00 & 3.00 & 3.01 & 4.00 & 4.00 & 5.00 & 5.82 & 6.00 & 7.61 \\ 
1/$2^{6}$ & 0.21 & 2.00 & 2.00 & 3.00 & 3.00 & 4.00 & 4.00 & 5.00 & 5.30 & 5.97 & 6.07 
\end{tabular}
\end{center}
\caption{Convergence results in the propagative zone, computed  at $G=[-3,1]$ with different basis functions. Comparison between classical Plane Waves and Generalized Plane Waves for both $\beta$ and constant normalizations. 
Some of the associated orders of convergence are also provided.
}\label{fig:cvP}
\end{figure}

\subsubsection{In the non propagative zone}
The point $G=[2,1]$ is in the non propagative zone. Again concentric disks are centered on $G$ with radius $h=1/2^k$, increasing the value of $k$, and the expected order of convergence is $n+1$. There is no classical Plane Wave that can be computed here since $\beta(G)>0$.

Figure \ref{fig:cvNP} displays computed convergence results that fit perfectly the theoretical result as well. Again machine precision is reached in some cases.
\begin{figure}
\begin{center}
\includegraphics{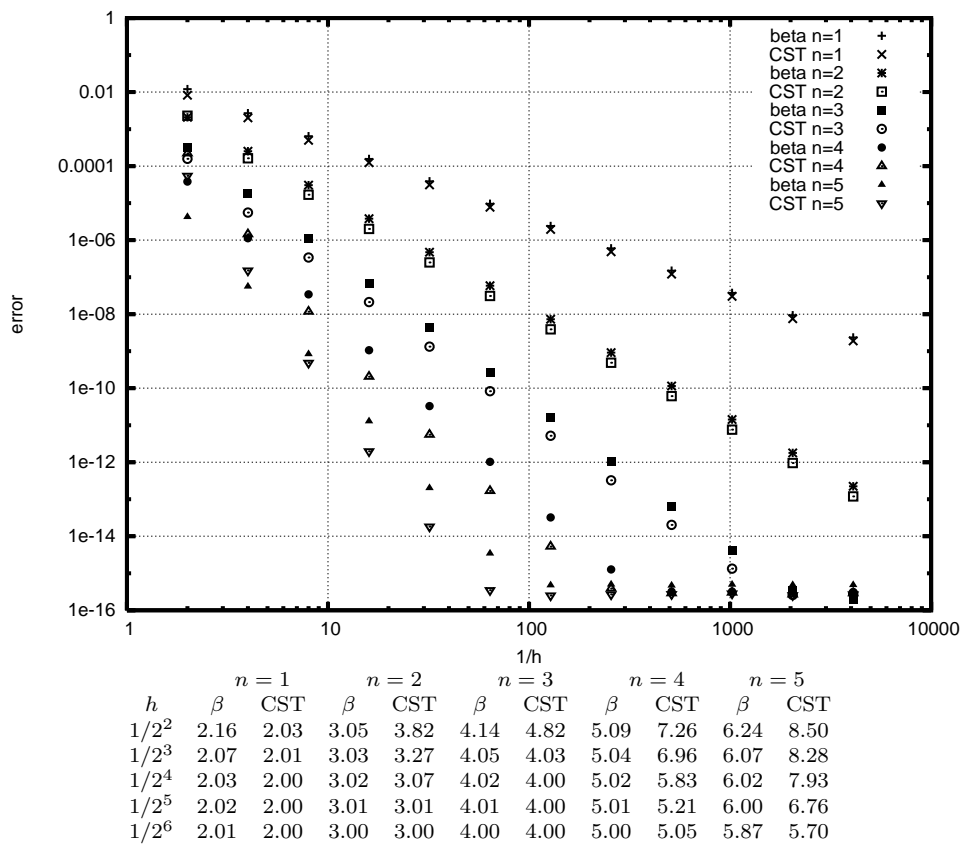}
\begin{tabular}
{@{}c@{\hspace{.3cm}}c@{\hspace{.3cm}}c@{\hspace{.3cm}}c@{\hspace{.3cm}}c@{\hspace{.3cm}}c@{\hspace{.3cm}}c@{\hspace{.3cm}}c@{\hspace{.3cm}}c@{\hspace{.3cm}}c@{\hspace{.3cm}}c@{}}
 \multicolumn{1}{c}{} & \multicolumn{2}{c}{$n=1$}& \multicolumn{2}{c}{$n=2$}& \multicolumn{2}{c}{$n=3$}& \multicolumn{2}{c}{$n=4$}& \multicolumn{2}{c}{$n=5$} \\
$h$ &  $\beta$ & CST & $\beta$ & CST & $\beta$ &CST & $\beta$ & CST & $\beta$ & CST\\ 
1/$2^{2}$  & 2.16 & 2.03 & 3.05 & 3.82 & 4.14 & 4.82 & 5.09 & 7.26 & 6.24 & 8.50 \\ 
1/$2^{3}$  & 2.07 & 2.01 & 3.03 & 3.27 & 4.05 & 4.03 & 5.04 & 6.96 & 6.07 & 8.28 \\ 
1/$2^{4}$  & 2.03 & 2.00 & 3.02 & 3.07 & 4.02 & 4.00 & 5.02 & 5.83 & 6.02 & 7.93 \\ 
1/$2^{5}$  & 2.02 & 2.00 & 3.01 & 3.01 & 4.01 & 4.00 & 5.01 & 5.21 & 6.00 & 6.76 \\ 
1/$2^{6}$  & 2.01 & 2.00 & 3.00 & 3.00 & 4.00 & 4.00 & 5.00 & 5.05 & 5.87 & 5.70 
\end{tabular}
\end{center}
\caption{Convergence results in the non-propagative zone, computed  at $G=[2,1]$ with different basis functions. Comparison between Generalized Plane Waves for $\beta$ and constant normalizations. 
Some of the associated orders of convergence are also provided.
}\label{fig:cvNP}
\end{figure}

\subsubsection{Toward the cut-off : $\beta\rightarrow 0$}
Since $\beta(x=1,y)=0$, it is interesting to look at what happens with the $\beta$-normalization along this line. Again the value of $h$ is  $h=1/2^k$, increasing the value of $k$. The point $G_h=[1-h,1]$ remains in the propagative zone. Then disks are here centered on a point $G_h$ that stands at a distance $h$ from the line $x=1$, still with radius $h$. As a result all the disks are tangent to the cut-off line defined by $x=1$. Classical Plane Waves are compared to the $\beta$-normalization with the same number of basis functions. 

Figure \ref{fig:cvtoCO} show that the $\beta$ normalized Generalized plane waves give a high order approximation of $u$ even getting closer to the vanishing line $x=1$, as long as $h$ is not too small. Note that  there does not seem to be a significant difference between the two type of functions. This is observed on numerical results even if there is no corresponding theoretical result.
\begin{figure}
\begin{center}
\includegraphics{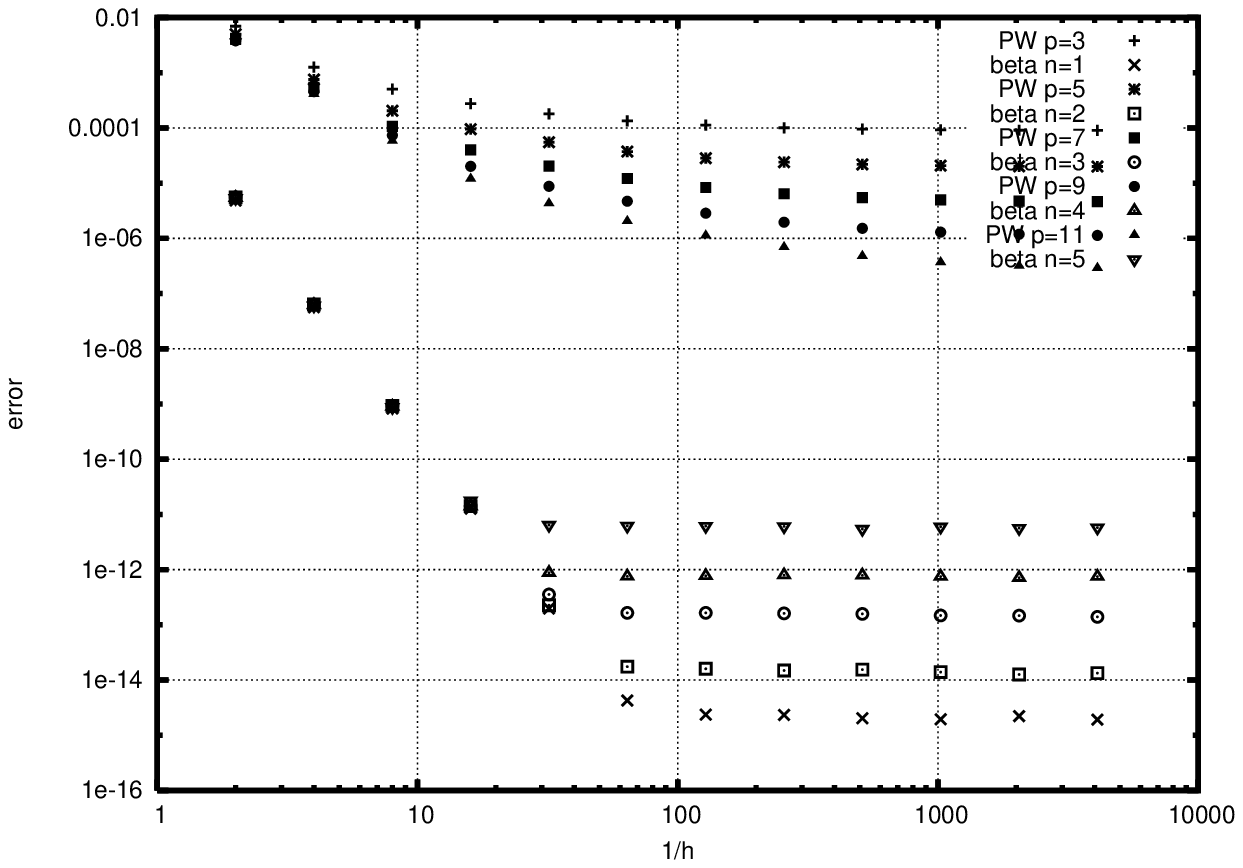}
\begin{tabular}
{@{}c@{\hspace{.3cm}}c@{\hspace{.3cm}}c@{\hspace{.3cm}}c@{\hspace{.3cm}}c@{\hspace{.3cm}}c@{\hspace{.3cm}}c@{\hspace{.3cm}}c@{\hspace{.3cm}}c@{\hspace{.3cm}}c@{\hspace{.3cm}}c@{}}
 \multicolumn{1}{c}{} & \multicolumn{2}{c}{$n=1$}& \multicolumn{2}{c}{$n=2$}& \multicolumn{2}{c}{$n=3$}& \multicolumn{2}{c}{$n=4$}& \multicolumn{2}{c}{$n=5$} \\
$h$  & PW &  $\beta$ & PW & $\beta$ & PW & $\beta$ &PW & $\beta$ & PW & $\beta$\\ 
1/$2^{2}$  & 2.35 & 2.37 & 3.27 & 3.24 & 3.17 & 4.09 & 3.21 & 5.32 & 3.21 & 6.25 \\ 
1/$2^{3}$  & 2.23 & 2.24 & 3.19 & 3.19 & 2.77 & 4.13 & 2.82 & 5.14 & 2.82 & 6.11 \\ 
1/$2^{4}$  & 2.14 & 2.15 & 3.12 & 3.11 & 2.47 & 4.09 & 2.40 & 5.08 & 2.40 & 6.03 \\ 
1/$2^{5}$  & 2.08 & 2.09 & 3.08 & 3.07 & 2.27 & 4.06 & 2.24 & 5.05 & 2.24 & 1.13 \\ 
1/$2^{6}$  & 2.05 & 2.04 & 3.05 & 3.04 & 2.15 & 4.04 & 2.13 & 4.00 & 2.13 & -2.24 
\end{tabular}
\end{center}
\caption{Convergence results toward $\beta=0$, computed  at $G_h=[1-h,1]$. Comparison between Classical Plane Waves and Generalized Plane Waves with the $\beta$-normalization. 
Some of the associated orders of convergence are also provided.
}\label{fig:cvtoCO}
\end{figure}

Another possibility is to compare the influence of two parameters : the size of the disk $h$ and the distance $d$ between $G$ and the line $x=1$. The error $e$ depends on both parameters, so one can write $e(h,d)$. Figure \ref{fig:cvhd} displays the error computed for $h$ and $d$ convergence with the $\beta$-normalization. The $h$ convergence is clearly damaged for decreasing values of $d$. This is linked to the low frequency limit when $\beta$ goes to zero. However, looking at the $h$ convergence with $d=h$, one can see that the error $e(h,h)$ converges as the error $e(h,1/2)$ until $h=1/2^5$.
\begin{figure}
\begin{center}
\begin{tabular}
{@{}c@{\hspace{.3cm}}c@{\hspace{.3cm}}c@{\hspace{.3cm}}c@{\hspace{.3cm}}c@{\hspace{.3cm}}c@{\hspace{.3cm}}c@{\hspace{.3cm}}c@{\hspace{.3cm}}c@{\hspace{.3cm}}c@{\hspace{.3cm}}c@{\hspace{.3cm}}c@{\hspace{.3cm}}c@{}}
h$\backslash$ d & 1/$2^{1}$& 1/$2^{2}$& 1/$2^{3}$& 1/$2^{4}$& 1/$2^{5}$& 1/$2^{6}$& 1/$2^{7}$& 1/$2^{8}$& 1/$2^{9}$& 1/$2^{10}$\\ 
  1/$2^{1}$  & 4.8e-06 & 5.5e-06 & 5.5e-06 & 5.4e-06 & 5.4e-06 & 5.3e-06 & 5.2e-06 & 5.2e-06 & 5.2e-06 & 5.2e-06 \\ 
1/$2^{2}$  & 5.7e-08 & 6.4e-08 & 6.4e-08 & 6.2e-08 & 6.1e-08 & 6.0e-08 & 5.9e-08 & 5.8e-08 & 5.8e-08 & 6.9e-08 \\ 
1/$2^{3}$  & 8.3e-10 & 9.2e-10 & 9.2e-10 & 9.0e-10 & 8.8e-10 & 8.7e-10 & 9.2e-10 & 1.2e-09 & 3.5e-09 & 2.4e-08 \\ 
1/$2^{4}$  & 1.3e-11 & 1.4e-11 & 1.4e-11 & 1.4e-11 & 1.8e-11 & 3.6e-11 & 1.0e-10 & 5.4e-10 & 3.2e-09 & 2.2e-08 \\ 
1/$2^{5}$  & 2.0e-13 & 2.3e-13 & 3.5e-13 & 8.8e-13 & 6.4e-12 & 2.8e-11 & 1.2e-10 & 5.4e-10 & 3.8e-09 & 2.2e-08 \\ 
1/$2^{6}$  & 4.3e-15 & 1.7e-14 & 1.6e-13 & 7.6e-13 & 6.2e-12 & 3.0e-11 & 1.0e-10 & 6.0e-10 & 3.1e-09 & 2.0e-08 \\ 
1/$2^{7}$  & 2.4e-15 & 1.6e-14 & 1.6e-13 & 7.7e-13 & 6.2e-12 & 2.8e-11 & 9.8e-11 & 5.1e-10 & 2.9e-09 & 2.3e-08 \\ 
1/$2^{8}$  & 2.3e-15 & 1.5e-14 & 1.6e-13 & 7.9e-13 & 6.1e-12 & 2.7e-11 & 1.0e-10 & 5.0e-10 & 2.5e-09 & 1.6e-08 \\ 
1/$2^{9}$  & 2.0e-15 & 1.5e-14 & 1.6e-13 & 7.9e-13 & 5.4e-12 & 2.5e-11 & 9.7e-11 & 4.9e-10 & 2.5e-09 & 1.9e-08 \\ 
1/$2^{10}$  & 1.9e-15 & 1.4e-14 & 1.5e-13 & 7.5e-13 & 6.0e-12 & 2.5e-11 & 8.7e-11 & 5.0e-10 & 2.5e-09 & 1.8e-08 \\ 
\end{tabular}
\end{center}
\caption{Error computed on a disk of radius $h$ centered at $G=[1-d;1]$. The approximation is computed with $\beta$-normalized basis functions and with $n=5$.}
\label{fig:cvhd}
\end{figure}

\subsubsection{Along the cut-off :  $\beta=0$}
The point $G=[1,1]$ lies exactly on the vanishing line of $\beta$.  Then again concentric disks are centered on $G$ with radius $h=1/2^k$, increasing the value of $k$. Both classical Plane Waves and Generalized Plane Waves with $\beta$-normalization would provide only one function since $\beta(G)=0$. As to the Generalized Plane waves with constant-normalization, the theoretical results show that their interpolation property holds along the cut-off as well as anywhere else in the domain.

As Figures \ref{fig:cvP} and \ref{fig:cvNP}, Figure \ref{fig:cvCO} displays results that fit perfectly the theoretical result. It is an example of efficient approximation of the exact solution $u_e$ along the cut-off.
\begin{figure}
\begin{center}
\includegraphics{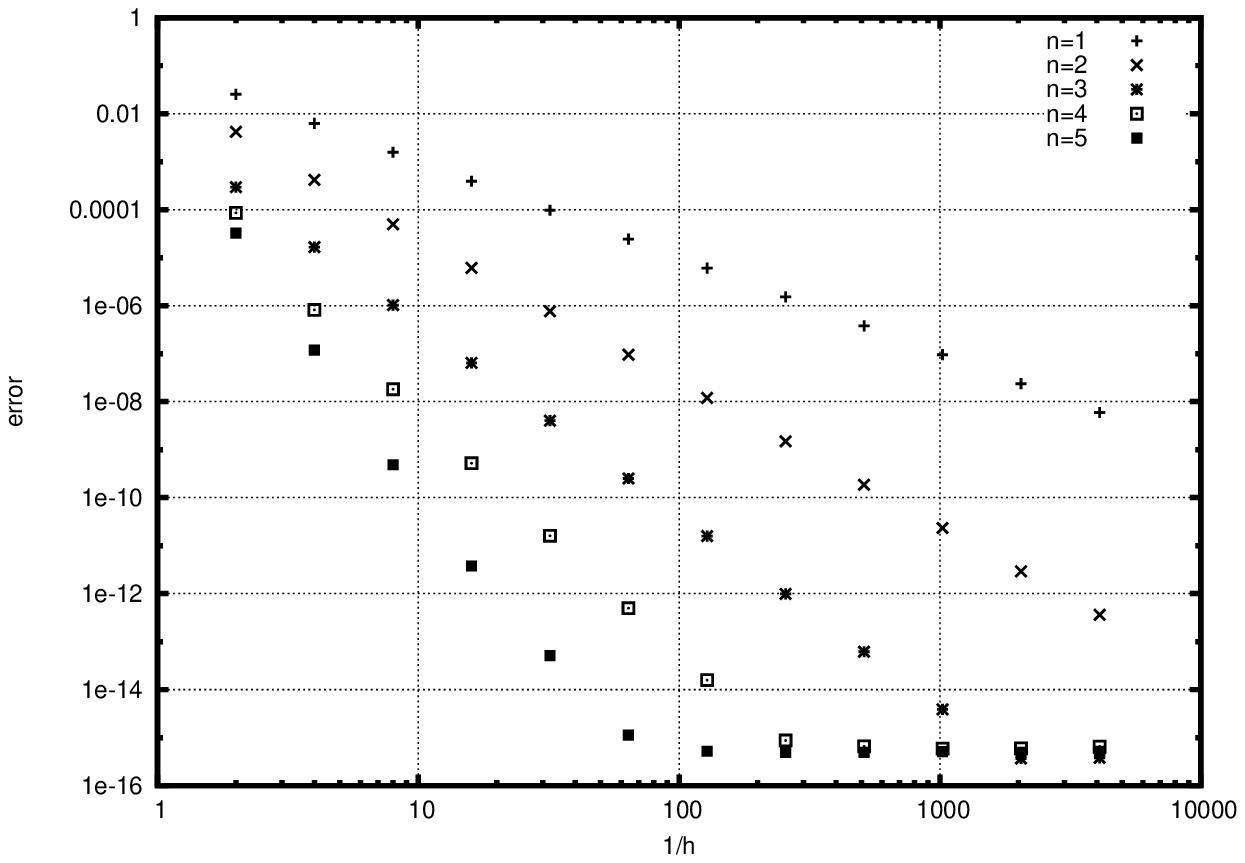}
\begin{tabular}
{@{}c@{\hspace{.3cm}}c@{\hspace{.3cm}}c@{\hspace{.3cm}}c@{\hspace{.3cm}}c@{\hspace{.3cm}}c@{}}
$h$ &  $n=1$ & $n=2$ &$n=3$ & $n=4$ & $n=5$\\ 
1/$2^{2}$  & 2.01 & 3.33 & 4.15 & 6.72 & 8.10 \\ 
1/$2^{3}$  & 2.00 & 3.09 & 4.01 & 5.49 & 7.93 \\ 
1/$2^{4}$  & 2.00 & 3.02 & 4.00 & 5.11 & 7.01 \\ 
1/$2^{5}$  & 2.00 & 3.00 & 4.00 & 5.03 & 6.20 \\ 
1/$2^{6}$  & 2.00 & 3.00 & 4.00 & 5.01 & 5.50 \\ 
\end{tabular}
\end{center}
\caption{Convergence results computed at $G=[1,1]$ where $\beta(G)=0$ using Generalized Plane Waves with the constant-normalization. 
Some of the associated orders of convergence are also provided.
}\label{fig:cvCO}
\end{figure}

\section{Application to the UWVF}\label{apps}

The Ultra-Weak Variational Formulation was proposed by B. Despr\'es in \cite{d94}. This numerical method is an example of Trefftz-based method and the idea to couple it with GPW was already proposed in \cite{LMIGBD}. This section presents the numerical method resulting from this coupling and includes corresponding numerical results.

\subsection{A generalized plane wave numerical method for smooth non constant coefficients
}
Associating to the scalar wave equation a boundary condition, consider now the following problem.
\begin{equation}\label{pbi}
\left\{
\begin{array}{rclc}
  -\Delta u+ \beta u&=&f,&(\Omega), \\
  \left(  \partial_\nu +i \gamma \right) u&=&Q\left( - \partial_\nu +i \gamma \right)u + g, &( \Gamma).
 \end{array}
\right. 
\end{equation}

\subsubsection{Notation}

Consider a domain $\Omega\subset \mathbb R^2$, with a mesh $\displaystyle \Omega = \bigcup_{k_1}^{N_h} \overline{\Omega_k }$, 
where the boundary $\bOk$ is of class $C^1$ almost everywhere. 
Let $h_k$ be the diameter of $\Ok$ and $\rho_k$ be the maximum of the diameters of
the spheres inscribed in $\Ok$. The mesh is such that
$\exists \sigma$ such that $h_k \leq \sigma \rho_k$. The refinement parameter or mesh size parameter $h$ is then defined by 
$h = \max h_k$. The terminology “regular mesh” comes from \cite{ciarlet}.

On a given element of the mesh $\Omega_k$, the center of gravity is denoted $G_k$, and $p(k)$ is the number of basis functions on $\Omega_k$. These basis functions $\{\varphi_k^l\}_{l\in[\![1,p(k)]\!]}$ are set to be zero on $\Omega\backslash \Omega_k$. This process defines a set of basis functions on $\Omega$, namely
\enn{
\mathcal E = \cup_k\mathcal E (G_k,N,p(k),q) \textrm{ where }\mathcal E (G_k,N,p(k),q)=\left\{ \varphi_k^l \right\}_{l\in[\![1,p(k)]\!]}.
}

The function space for the UWVF is denoted $V$ as 
\begin{equation*}
 V=\prod_{k \in \unNh} L^2(\bOk),
\end{equation*}
equipped with the Hermitian product
$
 \left( X,Y \right)= \sum_k \int_{\bOk} 
\frac1\gamma
X_k \overline{Y_k}$. 
It  defines a norm:  $\|X \|=\sqrt{(X,X)  }$.
In particular for any  operator  $A \in \mathcal L(V)$, the norm is 
\begin{equation*}
\|A\| =\sup_{X\neq 0} \frac{\|AX\|}{\|X\|}.
\end{equation*}
 As a consequence any element of $V$ is actually defined on the edges of the mesh elements.

\subsubsection{Adapted Ultra Weak Variational Formulation}
\label{ssec:aUWVF}

Let us start from the classical ultra weak Variational formulation.
The test functions space is defined by
\begin{equation}\label{Hk}
 H=\prod_{k=1}^{N_h} SpanH_k(\beta) \textrm{ where }H_k (\beta)= \left\{ v_k \in H^1(\Ok),\left| \begin{array}{l}
                                        (-\Delta +  \beta)v_k = 0, (\Ok),
					\\ \left( 
\trr v_k\right)_{|\bOk} \in L^2(\bOk)
                                       \end{array}
 \right. \right\}.
\end{equation}

\begin{theorem}\label{thuwvf}
 Let $u\in H^1 (\Omega)$ be a solution of problem \eqref{pbi} such that $\partial_{\nu_k} u \in L^2(\bOk)$ for any k.
Let $\gamma>0$ be a given real number.
 Then $X\in V$
 defined by $X_{|\bOk}=X_k$ with
 $X_k=((-\partial_\nu + i \gamma)u_{|\Omega_k})_{|\bOk}$ satisfies 
\begin{equation}\label{UWVF}
\begin{array}{ll}
\displaystyle
 \sum_k \left( \int_{\bOk} \frac{1}{\gamma} X_k
 \overline{\trr e_k}
- \sum_{j,j\neq k} \int_{\Sigma_{kj}} \frac{1}{\gamma} X_j
 \overline{(\partial_\nu + i \gamma)e_k} \right) 
\\ 
\displaystyle
 -\sum_{k, \Gamma_k\neq \emptyset} \int_{\Gamma_k} \frac{Q}{\gamma} X_k \overline{(\partial_\nu + i \gamma)e_k}
=-2i \sum_k \int_{\bOk} f \overline e + \sum_k \int_{\Gamma_k} \frac{1}{\gamma}g\overline{(\partial_\nu + i \gamma)e_k}, 
\end{array}
\end{equation}
for any 
$
 e= (e_k)_{k\in \unNh}\in H$.
 Conversely, if $X\in V$ is solution of $\eqref{UWVF}$ then the function $u$ defined locally  by
\begin{equation} \label{eqAA}
\left\{
\begin{array}{l}
 u_{|\Omega_k}= u_k \in H^1(\Omega_k), \\
 (-\Delta +\beta) u_k = f_{|\Omega_k}, \\
(-\partial_{\nu_k} + i \gamma) u_k = X_k,
\end{array}
\right.
\end{equation}
is the unique solution of the problem \eqref{pbi}.
\end{theorem}
This result is classical in the context of UWVF. We refer
to \cite{cd98,monk:buffa,hip1,monk2,monk3}. 
Even though the formalism and the use of approximated solutions as test functions has been described in \cite{LMIGBD}, the compact formulation is once again introduced hereafter, together with its adaptation to the generalized plane waves.
\begin{definition}
For any $f\in L^2(\Omega)$, let $E_f$ be the extension mapping defined by :
\begin{equation*}
 E_f: \left\{
\begin{array}{rcl}
 V &\rightarrow& H,
\\ Z & \mapsto& e=(e_k)_{k\in \unNh},
\end{array}
\right.
\text{ where }\forall k \in \unNh, \left\{
\begin{array}{rcl}
 (-\Delta +
 \beta) e_k &= 0 &(\Ok), \\
(-\partial_{\nu_k} + i \gamma) e_k &= Z_k &(\bOk).
\end{array}
\right.
\end{equation*}
Also define  $E$ which is the homogeneous extension
operator with 
vanishing right hand side, namely
 $E=E_0$.

 Let $F$ and $\Pi$ be the mapping defined by
\begin{equation*}
 F: \left\{
\begin{array}{rcl}
 V &\rightarrow& V,
\\ Z & \mapsto& \big( \trs E(Z)_{|\bOk} \big)_{k\in \unNh}.
\end{array}
\right.
\end{equation*}
\begin{equation*}
 \Pi: \left\{
\begin{array}{ccc}
 V &\rightarrow& V,
\\ Z_{|\skj} & \mapsto& Z_{|\sjk},
\\ Z_{|\Gk} & \mapsto & Q Z_{|\Gk}.
\end{array}
\right.
\end{equation*}
\end{definition}
Then, see \cite{cd98}, the problem \eqref{UWVF} is  equivalent to
\begin{equation}\label{eqPS}
 \left\{
\begin{array}{l}
 \text{Find }X\in V \text{ such that } \forall Y \in V
\\ (X,Y) - (\Pi X,F Y) = (B,Y),
\end{array}
\right.
\end{equation}
where the right hand side $b\in V$ is given by the Riesz theorem
\begin{equation*}
 (B,Y) = -2i 
 \int_{\Omega}
 f  \overline{E(Y)}
 +  \int_{\Gamma} \frac{1}{\gamma}g
\overline{F(Y)  }, \qquad \forall Y\in V.
\end{equation*}

The classical discretization process would be to consider $V_h\subset V$ of finite dimension and solve the problem \eqref{eqPS} on $V_h$. However, since the generalized plane waves do not belong to $H$, this formulation \eqref{eqPS} has to be adapted to the new basis functions. Indeed, there is no such thing as a unique equation satisfied by the basis functions to define an extension mapping from $V$ to $H$. The extension mapping is then defined as a one-to-one function from a subset of $V$ to the set of basis functions $\mathcal E$.

Some additional notations will be useful. The local discrete space is 
\begin{equation*}
 W_k = Span\left\{ \trr \varphi_k^l \right\}_{1\leq l \leq p(k)}
\subset L^2(\partial \Omega_k).
 \end{equation*}
The global discrete space $V^q  \subset V $ is defined by :
$
V^q=\prod_{1\leq k \leq N_h}  W_k$.
  It is therefore convenient to define the trace 
$v_k^l\in V$ by 
$$
V_k^l=\trr \varphi_k^l \mbox{ on } L^2(\partial \Omega_k), \quad
\mbox{ and }V_k^l=0\mbox{ on } L^2(\partial \Omega_{k'})\ \ k'\neq k.
$$
An equivalent  way to define $W_k$ and $V^q$ could be
$$
W_k=Span(V_k^l)_{1\leq l \leq p(k)}
\mbox{ and }V^q=Span(V_k^l)_{1\leq k\leq p(k), \ 1\leq p \leq N_h}
.
$$
Next define what are the generalizations
of operators $E$ and $F$ in this context.
 Let $E^q$ belong to ${\cal L}(V^q,\prod_{k=1}^{N_h} H^1(\Omega_k)$, be the discrete  mapping defined $\forall k \in \unNh$ and $\forall l \in [\![ 1, p(k) ]\!]$ by 
\begin{equation}\label{eq:Eq}
E^q(V_k^l)=
\varphi_k^l \mbox{ on } H^1(\Omega_k), \quad
\mbox{ and }V_k^l=0\mbox{ on } H^1( \Omega_{k'})\ \ k'\neq k.
\end{equation}
Similarly define
$F^q\in {\cal L}(V^q,V)$, $\forall k \in \unNh$ and $\forall l \in [\![ 1, p(k) ]\!]$, by 
\begin{equation*}
F^q(V_k^l)=
\trs( \varphi_k^l ) \mbox{ on } L^2(\partial \Omega_k), \quad
\mbox{ and }V_k^l=0\mbox{ on } L^2(\partial  \Omega_{k'})\ \ k'\neq k.
\end{equation*} 



\begin{remark}
From the definition of $P_\Delta$ one can see that each one of the basis functions $\varphi_l\in\mathcal E (G,N,p,q)$ is solution to a different equation, namely
\enn{\left(-\Delta+P_{\Delta,l}\right) \varphi_l=0}
where $P_{\Delta,l}=\dx^2P_l+\dy^2 P_l+(\dx P_l)^2+(\dy P_l)^2$.
Each of these equations are different, but they satisfy for each $l$: $P_{\Delta,l}-\beta=O(h^q)$.
\end{remark}

With these notations and definitions, the abstract
UWVF with generalized plane waves is defined as follows.
\begin{definition}{\bf (UWVF method  with generalized plane waves)}
Find $X_h \in V^q$ such that
 \begin{equation} \label{eq:mmp} 
 \forall Y_h \in V^q, (X_h,Y_h)_V - (\Pi X_h, F^q Y_h)_V = (B^q,Y_h)_V
 \end{equation}
with the right hand side given by
\e{\label{eq:bq}
(B^q,Y_h)_V=
-2i 
 \int_{\Omega}
 f  \overline{E^q(Y_h)}
 +  \int_{\Gamma} \frac{1}{\gamma}g
\overline{F^q(Y_h)  }, \qquad \forall Y_h\in V^q.
}
\end{definition}

\subsubsection{Interpolation interpretation}

$P_h$ denotes the orthogonal projection in $V$ on $V^q$.
\begin{proposition}
Let $u$ be a solution of a homogeneous scalar wave equation problem. Assume that $u$ is of class $C^{n+1}$ with $n \geq 1$. Let $X\in V$ satisfy
$X_{|\bOk}= (-\partial_{\nu_k} + i\omega)u_{|\bOk}$.
 The number of basis functions $Z_{k,l}=\trr \varphi_k^l$ per element $\Ok$ is fixed $p = 2n+1$. 
Let us assume, for the sake of the simplicity of
the proof, that their directions are fixed once for all. Then
$\exists C > 0$ depending on $n$ and the problem’s data $\gamma$ such that
$$\|(I - P_h )X\|_V ≤ Ch^{n-1/2} \|u\|_{C^{n+1}(\Omega)} . $$
\end{proposition}
\begin{proof}
Applying Theorem \ref{th:u-ua} on every element of the mesh, one gets $\exists u_a \in \mathcal E$, $u_a=\sum_{k,l} x_k^l\varphi_k^l$ such that $\forall \overrightarrow x\in \Omega$:
\begin{equation*}
\left\{
\begin{array}{l}
\left| u(\overrightarrow x)-u_a(\overrightarrow x)\right| \leq C h^{n+1} \left\| u \right\|_{\mathcal C^{n+1}(\Omega)} ,
\\\nonumber \left\| \nabla u(\overrightarrow x)-\nabla u_a(\overrightarrow x)\right\| \leq C h^{n} \left\| u \right\|_{\mathcal C^{n+1}(\Omega)},
\end{array}
\right.
\end{equation*}
where $C$ depends on $\Omega$. If $X_a\in V$ is defined by $(X_a)_{|\bOk}= (-\partial_{\nu_k} + i\omega)(u_a)_{|\bOk}$, then 
\enn{\tab{rl}{
\|X-X_a\|^2_{L^2(\bOk)} &\displaystyle \leq 2\int_\bOk \|\nabla u - \nabla u_a\|^2 + 2 \gamma^2\int_\bOk |u-u_a|^2
\\ &\displaystyle  \leq 2C^2h^{2n}\int_\bOk (1+\gamma^2h^2) \|u\|_{C^{n+1}(\Omega)},
}}
so that for $h$ small enough
\enn{
\|X-X_a\|\leq C h^{n+1/2} \sqrt{N_h} \|u\|_{C^{n+1}(\Omega)}.
}
The results then stems from the fact that, for a regular mesh, the total number of elements can be bounded by $C/h^2$.
\end{proof}


\subsection{Numerical results}
The numerical method presented in Subsection \ref{ssec:aUWVF} has been implemented in two dimensions. The main difference with a classical UWVF code lies in the need for a quadrature formula. Indeed, there is no exact integration formula when the basis functions are no more classical Plane Waves. Results corresponding to two Newton-Cotes formulas will be compared: a Boole formula with five points, a Weddle formula with seven points and a formula with ten points.

\subsubsection{Comparing $h$ convergence results on a test case}
Consider the system
\begin{equation}
\left\{\begin{array}{l}
-\Delta u + (x-1) u = 0,
\\(-\partial_\nu + i \gamma )u=g,
\end{array}\right.
\end{equation}
on the domain $[-6,3]\times[-1,1]$. This problem presents a cut-off along the line $x=1$. The boundary condition corresponds to the exact solution $u_e = Airy(-x) e^{iy}$. An example of computed solution is represented on Figure \ref{fig:sol10SQ75}.

\begin{figure}
\begin{center}
\includegraphics[height=6.5cm]{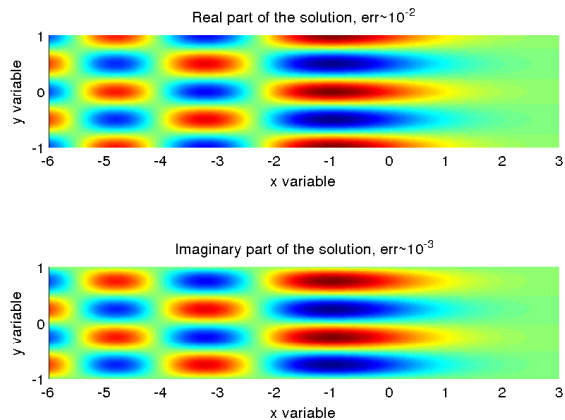}
\end{center}
\caption{The solution approximating $u_e=Airy(x)e^{iy}$ on the computational domain, computed with the 10 points quadrature formula. The basis functions used for the computation are built with the $\beta$-normalization using $p=7$ basis functions per element and the order of approximation $q=5$.}
\label{fig:sol10SQ75}
\end{figure}

Figure \ref{fig:compQF} evidences the limitation of low order quadrature formulas. It compares results obtained with Boole (5 points), Weddle (7 points) and the 10 points quadrature formulas. All the results are computed with $p=2n+1$ basis functions per element and with the order of approximation $q=n+1$. The results obtained with $n=3$ are plotted with different cross-marks, whereas those obtained with $n=4$ are plotted with different square- and disk-marks. 
For $n=3$ the choice of quadrature makes obviously no difference. However, to obtain a higher order convergence, there is a price to pay regarding quadrature formula. Indeed, for $n=4$ the situation is different : the computations performed with Boole formula do not reach the high convergence rate obtained with both Weddle and the 10 points formulas. This is due to the pollution introduced by the low order quadrature formula.
\begin{figure}
\begin{center}
\includegraphics[height=6.5cm]{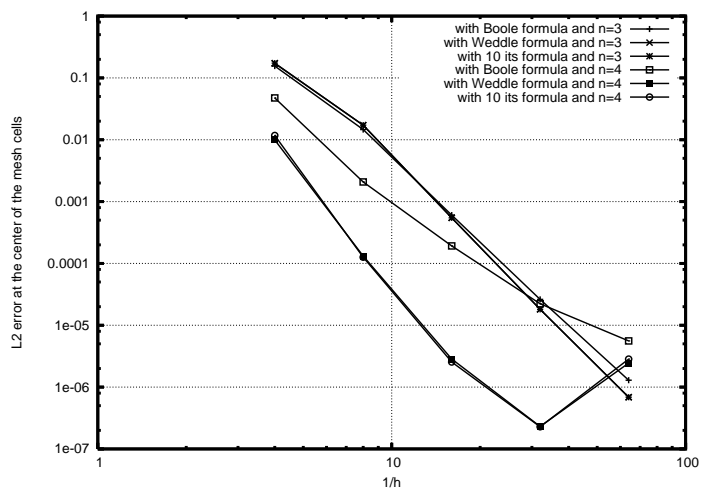}
\end{center}
\caption{$h$-convergence results displayed with different cross marks for $n=3$ and disk and square marks for $n=4$. The solution is computed with $p=2n+1$ basis functions per element with the order of approximation $q=n+1$. Results obtained with the three different quadrature formula are displayed, namely Boole, Weddle and the 10 points formula. These results are computed with the $\beta$-normalization. The error is represented with respect to $N_h = 1/h$ where $h$ stands for the size of the mesh.}
\label{fig:compQF}
\end{figure}

\begin{figure}
\begin{center}
\includegraphics[height=5cm]{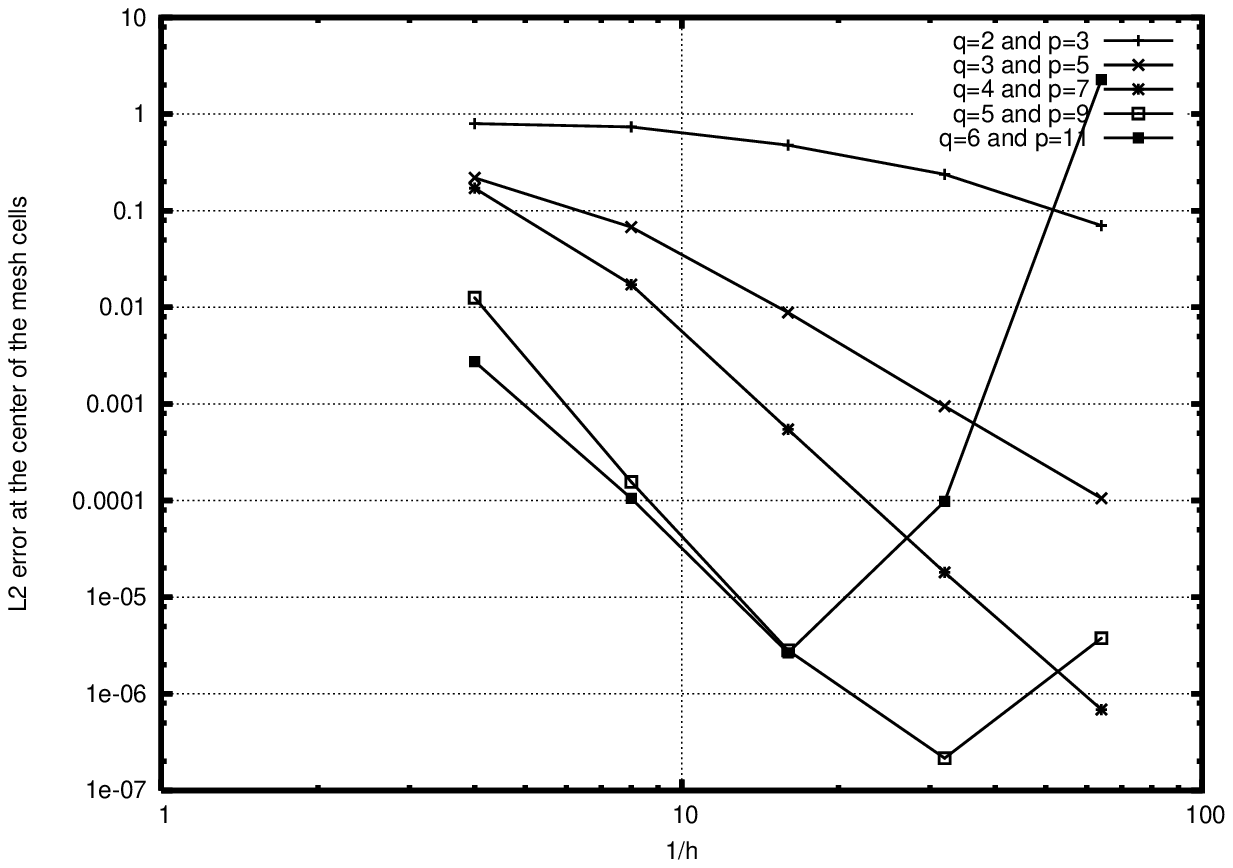}
\includegraphics[height=5cm]{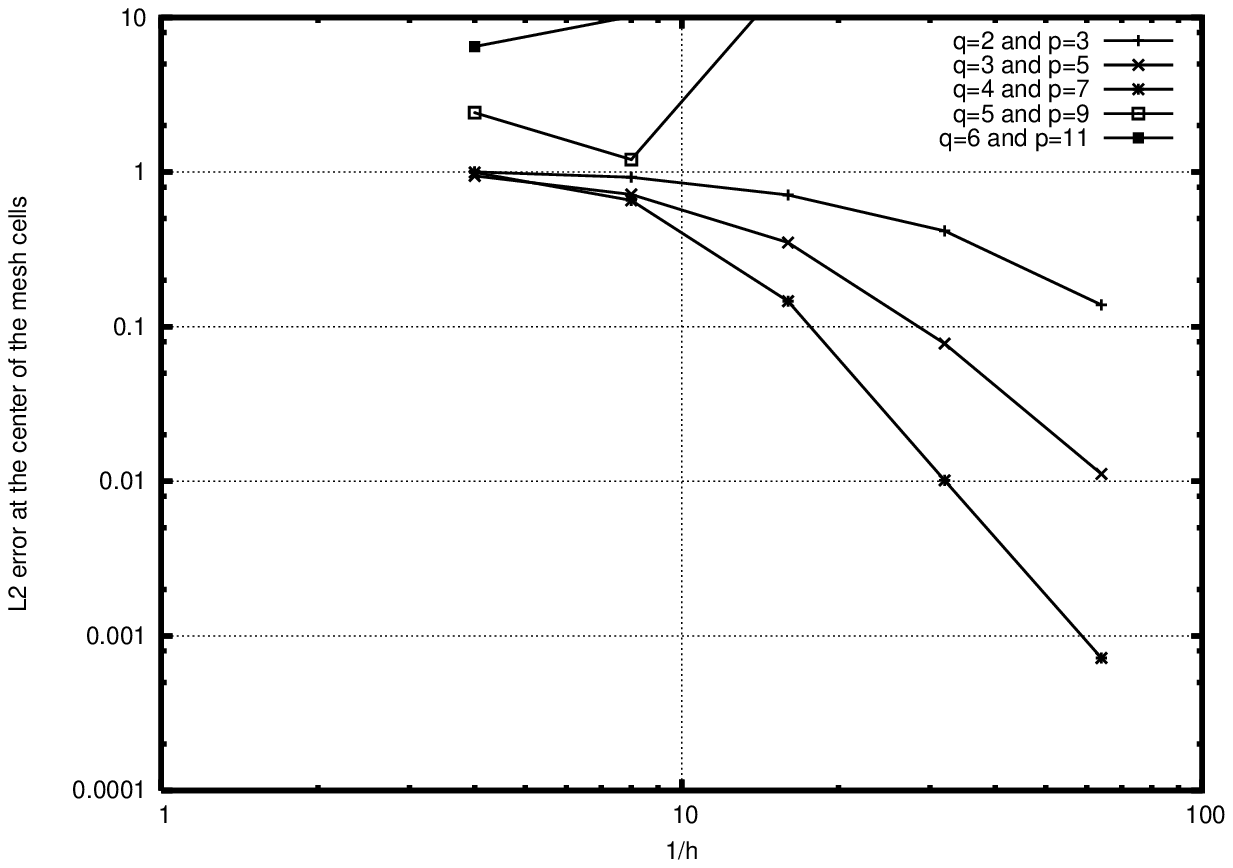}
\end{center}
\caption{$h$-convergence results displayed for $\beta$-normalization on the left and constant-normalization on the right, both using Weddle quadrature formula to compute the integrals. The relative discrete $L^2$ error computed at the centers of the mesh cells is represented depending on the inverse of the mesh size $h$. The results are computed with $p=2n+1$ and $q=n+1$, for $n$ between 1 and 5.}
\label{fig:Weddlecv}
\end{figure}
Figure \ref{fig:Weddlecv} compares convergence results with respect to the mesh size $h$, all the integral being computed with a Weddle quadrature formula. Results obtained with the $\beta$-normalization and with the constant-normalization are displayed respectively on the left and on the right, at the same scale. As suggested by the interpolation theorem, all the results are computed with $q=n+1$ and $p=2n+1$ where $n\in \mathbb N$. The left hand part of the figure displays results computed with the $\beta$-normalization. 
 For $p=11$, the ill-conditioning of the system matrix prevent the method from converging. It is remarkable that even though the basis functions with the $\beta$-normalization do not exist along the cut-off, this does not prevent the method from converging. The right hand side of the figure displays results computed with the constant-normalization. It shows that even if the constant-normalization ensures high order interpolation along the cut-off, the numerical method obtained by described in this section provides much smaller error when used with the $\beta$-normalized basis functions than with the constant-normalized basis functions. It also seems that the method with the constant-normalization is inadequate when the number $p$ of basis functions per element increases.

\subsubsection{A first physical test case}
This case was proposed by St\'ephane Heuraux as a second step toward reflectometry applications, for which the presence of a plasma cut-off is crucial. It models a wave sent in a plasma by an antenna from the wall of a fusion reactor, see \cite{LMth}. The antenna is represented by a wave guide added outside the reactor on a wall plus a horn inside the reactor. 

The geometry is described in Figure \ref{fig:domain}. The domain $\Omega$ is a $L\times L$ square, the  width and length of the waveguide are $l_0$ and $4l_0$. The size of the domain is set to be $L=50 l_0$, where $l_0$ is the wavelength of the incoming signal in the horn.  The cut-off is set at $x=40l_0$. The heterogeneous medium is modeled by the coefficient
\e{\beta(x)=
\left\{
\tab{ll}{
-\kappa^2,&x<2, \\
 -\kappa^2 (x-4)/(2),&x\geq 2.
}
\right.
}
One gets a wave propagating from the wave guide through the horn toward the right end of the domain, reflected by the cut off situated at $x=4$. See Figure \ref{fig:WaveCut}. This result was computed with the normalization $N=\sqrt{\beta}$.
\begin{figure}
\begin{center}
\includegraphics[height=6cm]{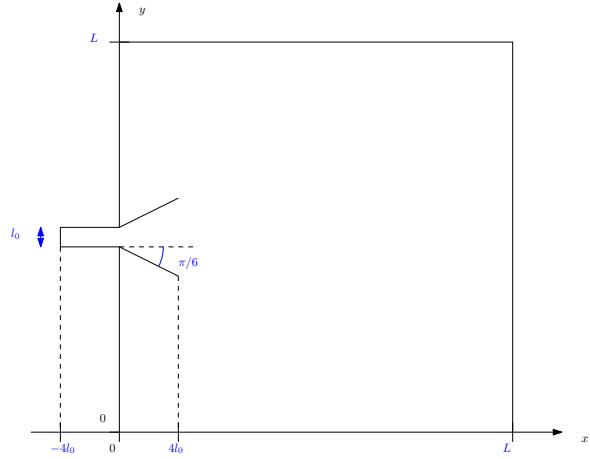}
\end{center}
 \caption{Slice of tokamak, specifying the domain parameters: the wave guide width, the shape of the horn and the size of the main part of the domain.}
\label{fig:domain}
\end{figure}
\begin{figure}
\begin{center}
\begin{tabular}{c}
\includegraphics[width=.8\textwidth]{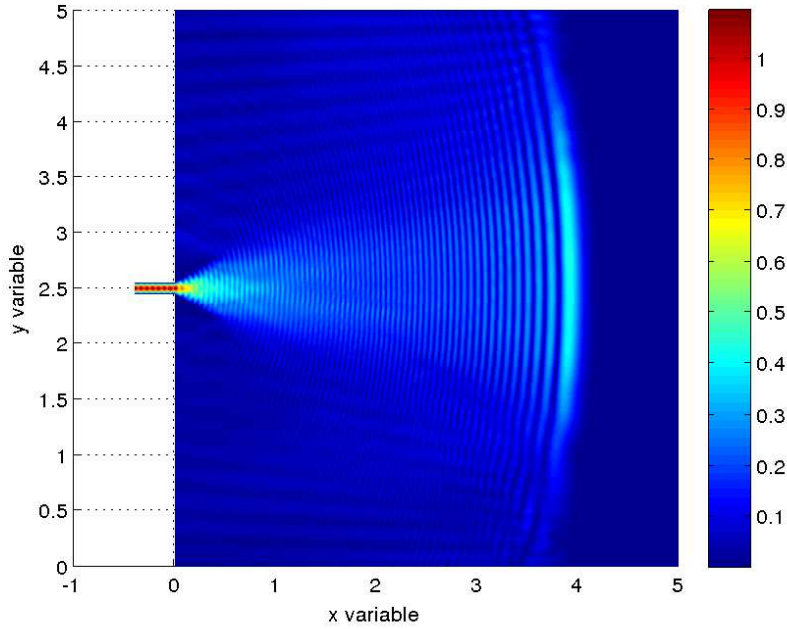}
 \end{tabular}
 \caption{Wave reflected by the cut off. Result computed using generalized plane waves designed with $N=\sqrt{\beta}$ and the UWVF, for $p=7$ and $q=4$.
Modulus of the computed solution.}
\label{fig:WaveCut}
\end{center}
\end{figure}

\section{Conclusion}

A procedure to design a set of generalized plane waves that are locally approximated solution of the scalar wave equation with \underline{smooth non constant coefficient} has been successfully developed. Both theoretical and numerical results evidence the \underline{high order approximation}. 
This procedure can easily be generalized to many differential operators, as described in \cite{LMth}.
Moreover a natural idea would be to extend the generalization process from the phase to the amplitude of plane waves, by considering a looking for a shape function as $\varphi= Qe^P$ where $P$ and $Q$ are two polynomials.

These generalized plane waves have been coupled with the UWVF to obtain a numerical method adapted to problems with smooth vanishing coefficients. This resulting numerical method has been described, and the first numerical results are promising but no theoretical convergence is available yet. 
The behavior of the method along a cut-off requires further investigation, since the comparison between numerical results obtained with the two normalizations of the basis functions is very not yet explained theoretically.

\appendix

\section{Chain rule in dimension 1 and 2}
For the sake of completeness, this section is dedicated to describing the formula to derive a composition of two functions, in dimensions one and two. 
A wide bibliography about this formula is to be found in \cite{ma}. It is linked to the notion of partition of an integer or the one of a set. The 1D version is not actually used in this work but is displayed here as a comparison with a 2D version, mainly concerning this notion of partition.
\subsection{Faa Di Bruno Formula}\label{app:FDBf}
Faa Di Bruno formula gives the $m$th derivative of a composite function with a single variable. It is named after Francesco Faa Di Bruno, but was stated in earlier work of Louis F.A. Arbogast around 1800, see \cite{007}.

If $f$ and $g$ are functions with sufficient derivatives, then
\enn{
\frac{d^m}{dx^m}f(g(x)) = m!\sum f^{(\sum_k b_k)} (g(x)) \prod_{k=1}^{m} \frac{1}{b_k!}\left(  \frac{g^{(k)}(x)}{k!} \right)^{b_k},
}
where the sum is over all different solutions in nonnegative integers $(b_k)_{k\in[\![1,m]\!]}$ of $\sum_k k b_k = m$. These solutions are actually the partitions of $m$.

\subsection{Bivariate version}\label{app:bivFDB}
The multivariate formula has been widely studied, the version described here is the one from \cite{CS} applied to dimension $2$. A linear order on $\mathbb N^2$ is defined by: $\forall (\mu,\nu)\in\left(\mathbb N^2\right)^2$, the relation $\mu\prec\nu$ holds provided that 
\begin{enumerate}
\item $\mu_1+\mu_2<\nu_1+\nu_2 $; or
\item $\mu_1+\mu_2=\nu_1+\nu_2 $ and $\mu_1<\nu_1$.
\end{enumerate}
If $f$ and $g$ are functions with sufficient derivatives, then
\enn{
\dx^i \dy^j 
f(g(x,y)) =i!j! \sum_{1\leq\mu\leq i+j} f^{\mu}(g(x,y)) \sum_{s=1}^{i+j} \sum_{p_s((i,j),\mu)} \prod_{l=1}^s \frac{1}{k_l!}\left( \frac{1}{i_l!j_l!} \dx^{i_l}\dy^{j_l} 
 (g(x,y))\right)^{k_l},
}
where the partitions of $(i,j)$ are defined by the following sets: $\forall \mu\in[\![1,i+j]\!]$, $\forall s\in[\![1,i+j]\!]$,
\enn{
p_s((i,j),\mu)=\left\{ (k_1,\cdots,k_s;(i_1,j_1),\cdots,(i_s,j_s)):k_i>0,0\prec(i_1,j_1)\prec\cdots\prec(i_s,j_s),\phantom{\sum_{l=1}^s }\right.}
\enn{\left. \sum_{l=1}^s k_l=\mu,\sum_{l=1}^s k_li_l=i,\sum_{l=1}^s k_lj_l=j\right\}.
}
See \cite{ardi} for a proof of the formula interpreted in terms of collapsing partitions.



\begin{acknowledgements}
I would like to thank Peter Monk for bringing to my attention the importance of such interpolation properties and for his hospitality during my visit to the University of Delaware, where I developed a 2D UWVF code with GPW for the scalar wave equation. I thank Teemu Luostari for providing his 2D PW-UWVF code for elasticity equations. This visit was funded by the Fondation Pierre Ledoux. I would also like to thank Bruno Despr\'es for his help.
\end{acknowledgements}


\end{document}